\documentclass[12pt]{amsart}
\usepackage{amssymb,amsmath,amsthm,mathrsfs,multirow,xcolor,framed,url}
\newtheorem{theorem}{Theorem}[section]
\newtheorem{lemma}[theorem]{Lemma}

\theoremstyle{definition}

\theoremstyle{remark}
\newtheorem*{remark}{Remark}

\numberwithin{equation}{section}

\newcommand\nutwid{\overset {\text{\lower 3pt\hbox{$\sim$}}}\nu}
\newcommand\newaa{\alpha}
\newcommand{\abs}[1]{\lvert#1\rvert}

\newcommand{\uhp}{\mathscr{H}}  

\newcommand\FL[1]{\left\lfloor#1\right\rfloor}

\newcommand\stroke[2]{{#1}\,\left\arrowvert\,#2\right.}

\newcommand\abcdMAT{\begin{pmatrix} a & b \\ c & d \end{pmatrix}}
\newcommand\GoneMAT{\begin{pmatrix} 1 & * \\ 0 & 1 \end{pmatrix}}
\newcommand\TMAT{\begin{pmatrix} 1 & 1 \\ 0 & 1 \end{pmatrix}}
\newcommand\MAT[4]{\begin{pmatrix} {#1} & {#2} \\ {#3}  & {#4} \end{pmatrix}}
\newcommand\Z{\mathbb{Z}}
\newcommand\C{\mathbb{C}}
\allowdisplaybreaks
\newcommand\omyeqn[1]{(\ref{eq:#1})}
\newcommand\omycite[1]{}
\newcommand\omylem[1]{\ref{lem:#1}}
\newcommand\omythm[1]{\ref{thm:#1}}
\newcommand\thm[1]{\ref{thm:#1}}
\newcommand\lem[1]{\ref{lem:#1}}
\newcommand\eqn[1]{(\ref{eq:#1})}
\newcommand\sect[1]{\ref{sec:#1}}

\newcommand\subsect[1]{\ref{subsec:#1}}

\newcommand\mylabel[1]{\label{#1}}
\newcommand\omylabel[1]{\label{#1}}

\newcommand\mythm[1]{\ref{thm:#1}}

\newcommand{\beqs}{\begin{equation*}}
\newcommand{\eeqs}{\end{equation*}}
\newcommand{\beq}{\begin{equation}}
\newcommand{\eeq}{\end{equation}}
\renewcommand{\MR}[1]{\href{http://www.ams.org/mathscinet-getitem?mr={#1}}{MR{#1}}}
\DeclareMathOperator{\IM}{Im}
\DeclareMathOperator{\ORD}{Ord}
\DeclareMathOperator{\ord}{ord}

\begin{document}
\title[Crank and rank parity function congruences]{Congruences modulo powers of $5$ and $7$\\
for the crank and rank parity functions and related mock theta functions}         

\author{Dandan Chen}
\address{Department of Mathematics, Shanghai University, People's Republic of China}
\email{mathcdd@shu.edu.cn}
\author{Rong Chen}
\address{School of Mathematical Sciences, Tongji University,
Shanghai, People's Republic of China}
\email{rchen@stu.ecnu.edu.cn}
\author{Frank Garvan}
\address{Department of Mathematics, University of Florida, Gainesville,
FL 32611-8105}
\email{fgarvan@ufl.edu}
\thanks{The first author was supported in part by the
National Natural Science Foundation of China (Grant No. 12201387) and Shanghai Sailing Program (\#21YF1413600).
The second author was supported in part by China Postdoctoral Science Foundation (\#2022M712422).
The third author was supported in part by a grant from
        the Simon's Foundation (\#318714).
The preliminary results
of this paper were first presented at the AMS Special Session
on Experimental and Computer Assisted Mathematics, Denver, January, 2020.
}


\date{\today}                  



\begin{abstract}
It is well known that Ramanujan conjectured congruences modulo powers of
$5$, $7$ and and $11$ for the partition function. These were subsequently
proved by Watson (1938) and Atkin (1967). In 2009 Choi, Kang, and
Lovejoy proved congruences modulo powers of $5$ for the crank parity
function.  The generating function for the analogous rank parity function
is $f(q)$, the first example of a mock theta function that Ramanujan mentioned
in his last letter to Hardy. Recently we proved congruences modulo 
powers of $5$ for the rank parity function, and here we extend these
congruences for powers of $7$.
We also show how  
these congruences  imply congruences modulo powers of $5$ and $7$ for the coefficients
of the related third order mock theta function $\omega(q)$, using
Atkin-Lehner involutions and transformation results of Zwegers.
Finally we a prove a family of congruences modulo powers of $7$ for
the crank parity function.
\end{abstract}

\maketitle

\section{Introduction}
\omylabel{sec:intro}
Let $p(n)$ be the number of unrestricted partitions of $n$.
Ramanujan discovered and later proved that
\begin{align}
p(5n+4) &\equiv 0 \pmod{5},\omylabel{eq:ram5} \\ 
p(7n+5) &\equiv 0 \pmod{7},\omylabel{eq:ram7} \\ 
p(11n+6) &\equiv 0 \pmod{11}.\omylabel{eq:ram11} 
\end{align}
In 1944 Dyson \cite{Dy44}\omycite{Dy44}
defined the \textit{rank} of a partition as
the largest part minus the number of parts and conjectured
that residue of the rank mod $5$ (resp. mod $7$) divides the partitions
of $5n+4$ (resp. $7n+5$) into $5$ (resp. $7$) equal classes thus giving
combinatorial explanations of Ramanujan's partition congruences mod $5$
and $7$. Dyson's rank conjectures were proved by
Atkin and Swinnerton-Dyer \cite{At-SwD}\omycite{At-SwD}.
Dyson also conjectured the existence of another statistic he called the
\textit{crank} which would likewise explain Ramanujan's partition
congruence mod $11$. The crank was found by Andrews and the third author
\cite{An-Ga88}\omycite{An-Ga88}
who defined the \textit{crank} as the largest part, if the partition has no
ones,
and otherwise as the difference between the number of parts larger than the
number of ones and the number of ones.

Let $M_e(n)$ (resp. $M_o(n)$) denote the number of partitions of $n$ with even (resp. odd) crank. Choi, Kang and Lovejoy \cite{Ch-Ka-Lo}\omycite{Ch-Ka-Lo} proved congruences
modulo powers of $5$ the \textit{crank parity function},
which is the difference
$$
M_e(n) - M_o(n).
$$
\begin{theorem}[{Choi, Kang and Lovejoy \cite[Theorem 1.1]{Ch-Ka-Lo}\omycite{Ch-Ka-Lo}}]
\omylabel{thm:crankthm}
For all $\newaa\ge0$ we have
$$
M_e(n) - M_o(n) \equiv 0 \pmod{5^{\newaa+1}},\qquad
\mbox{if $24n\equiv 1 \pmod{5^{2\newaa+1}}$}.
$$
\end{theorem}
This gave a weak refinement of Ramanujan's partition congruence modulo
powers of $5$:
$$
p(n) \equiv 0 \pmod{5^a},\qquad
\mbox{if $24n\equiv 1 \pmod{5^\newaa}$}.
$$
Ramanujan's partition congruence
was proved by Watson \cite{Wa38}\omycite{Wa38}. We define
\beq
\beta(n) = M_e(n) - M_o(n),
\mylabel{eq:betadef}
\eeq
for $n\ge0$.  We prove the following new 
\begin{theorem}
\mylabel{thm:crankthm7}
For each $\newaa\ge1$ there is an integral constant $K_\newaa$ such that
\beq
\beta(49n - 2) \equiv K_\newaa\, \beta(n) \pmod{7^{\newaa}},\qquad
\mbox{if $24n\equiv 1 \pmod{7^{\newaa}}$}.
\mylabel{eq:betamod7}
\eeq
\end{theorem}
This gives a weak refinement of Ramanujan's partition congruence modulo
powers of $7$:
$$
p(n) \equiv 0 \pmod{7^\newaa},\qquad
\mbox{if $24n\equiv 1 \pmod{7^{\FL{\tfrac{a+2}{2}}}}$}.
$$
This was also proved by Watson \cite{Wa38}. 
The congruence in \eqn{betamod7} is reminiscent of 
Atkin and O'Brien's \cite{At-OB}
congruences mod powers of $13$ for the partition function.

In \cite{Ch-Ch-Ga20} we considered analogues of 
Theorem \thm{crankthm} for the rank parity function.
Analogous to $M_e(n)$ and $M_o(n)$ we let $N_e(n)$ (resp. $N_o(n)$) denote the number of partitions of $n$ with even
(resp. odd) rank. It is well known that the difference is related to
Ramanujan's mock theta function $f(q)$. This is the first example of
a mock theta function that Ramanujan gave in his last letter to
Hardy. Let
\begin{align*}
  f(q) &= \sum_{n=0}^\infty a_f(n) q^n = 1 + \sum_{n=1}^\infty
\frac{q^{n^2}}{(1+q)^2(1+q^2)^2 \cdots (1+q^n)^2}\\
       &=
1+q-2\,{q}^{2}+3\,{q}^{3}-3\,{q}^{4}+3\,{q}^{5}-5\,{q}^{6}+7\,{q}^{7}-
6\,{q}^{8}+6\,{q}^{9}-10\,{q}^{10}+12\,{q}^{11}-11\,{q}^{12}+
 \cdots.
\end{align*}
This function has been studied by many authors. Ramanujan conjectured
an asymptotic formula for the coefficients $a_f(n)$. Dragonette
\cite{Dr52}\omycite{Dr52} improved this result by finding a Rademacher-type asymptotic
expansion for the coefficients. The error term was subsequently improved
by Andrews \cite{An66a}\omycite{An66a}, Bringmann and Ono \cite{Br-On06}\omycite{Br-On06}, and Ahlgren
and Dunn \cite{Ah-Du19}\omycite{Ah-Du19}. We have
$$
a_f(n) = N_e(n) - N_o(n),
$$
for $n\ge0$.

In \cite[Theorem 1.2 ]{Ch-Ch-Ga20}\omylabel{Ch-Ch-Ga20}, we proved
\begin{theorem}
\mylabel{thm:rankpar5}
For all $\alpha\ge3$ and all $n\ge 0$ we have
\begin{align}
\mylabel{eq:rmod5}
a_f(5^{\alpha}n + \delta_\alpha)
+ a_f(5^{\alpha-2}n + \delta_{\alpha-2})
\equiv 0 \pmod{5^{ \FL{\tfrac{1}{2}\alpha }}},
\end{align}
  where $\delta_\alpha$ satisfies $0 < \delta_\alpha < 5^\alpha$ and
$24\delta_\alpha\equiv1\pmod{5^\alpha}$.
\end{theorem}

In \cite{Ch-Ch-Ga20}, we also stated the following theorem without proof.
\begin{theorem}
\mylabel{thm:mainthm}
For all $\alpha\ge3$ and all $n\ge 0$ we have
\begin{align}
\mylabel{eq:rmod7}
a_f(7^{\alpha}n + \delta_\alpha)
- a_f(7^{\alpha-2}n + \delta_{\alpha-2})
\equiv 0 \pmod{7^{ \FL{\tfrac{1}{2}(\alpha-1) }}},
\end{align}
  where $\delta_\alpha$ satisfies $0 < \delta_\alpha < 7^\alpha$ and
$24\delta_\alpha\equiv1\pmod{7^\alpha}$.
\end{theorem}
The starting point of our proof of Theorem \thm{mainthm} is an
eta-product identity for the generating function of
$$
a_f(n/7)-a_f(7n-2).
$$
See Theorem \omythm{af7thm}. This enables us to
use the theory of modular functions to prove the  congruences. 
Our presentation and method
are similar to that Paule and Radu \cite{Pa-Ra12}\omycite{Pa-Ra12}, 
who solved a
difficult conjecture of Sellers \cite{Se1994} for congruences 
modulo powers of $5$
for Andrews's two-colored generalized Frobenius 
partitions \cite{An1984mem}.

The goal of this paper is prove Theorems \omythm{crankthm7} and 
\omythm{mainthm}, as well as prove analogous congruences for Ramanujan's
third order mock theta function $\omega(q)$:
\begin{align*}
\omega(q)=&\sum_{n=0}^{\infty}a_w(n)q^n=1+\sum_{n=1}^{\infty}\frac{q^{2n^2+2n}}{(1-q)^2(1-q^3)^2\cdots(1+q^{2n-1})^2}\\
=&1+q^4+2q^5+3q^6+4q^7+5q^8+6q^9+7q^{10}+8q^{11}+10q^{12}+\cdots.
\end{align*}
By utilising an Atkin-Lehner involution we show the following theorem
follows from Theorems \mythm{rankpar5} and \mythm{mainthm}.  

\begin{theorem} 
\mylabel{thm:mainthm-2}
  \begin{enumerate}
    \item[(i)]
For all $\alpha\ge3$ and all $n\ge 0$ we have
\begin{align}
\mylabel{eq:wrmod5}
a_\omega(5^{\alpha}n + \delta_\alpha)
+ a_\omega(5^{\alpha-2}n + \delta_{\alpha-2})
\equiv 0 \pmod{5^{ \FL{\tfrac{1}{2}\alpha }}},
\end{align}
  where $\delta_\alpha$ satisfies $0 < \delta_\alpha < 5^\alpha$ and
$3\delta_\alpha+2\equiv0\pmod{5^\alpha}$.
\item[(ii)]
For all $\alpha\ge3$ and all $n\ge 0$ we have
\begin{align}
\mylabel{eq:wrmod7}
a_\omega(7^{\alpha}n + \delta_\alpha)
+ a_\omega(7^{\alpha-2}n + \delta_{\alpha-2})
\equiv 0 \pmod{7^{ \FL{\tfrac{1}{2}(\alpha-1) }}},
\end{align}
  where $\delta_\alpha$ satisfies $0 < \delta_\alpha < 7^\alpha$ and
$3\delta_\alpha+2\equiv0\pmod{7^\alpha}$.
  \end{enumerate}
\end{theorem}

\begin{remark}
We note that Karl-Heine Fricke \cite{Fr13} independently observed 
 \omyeqn{rmod5}--\omyeqn{wrmod7} but without proof.
\end{remark}

In Section \sect{modfuncs} we include the necessary background and 
algorithms from the theory
of modular functions for proving identities. In 
Sections \sect{rankparity5}--\sect{w-5-7} we apply the
theory of modular functions to prove our 
Theorems \thm{crankthm7}, \thm{mainthm} and \thm{mainthm-2}.

\subsection*{Some Remarks and Notation}
\omylabel{subsec:notation}
Throughout this paper we use the standard $q$-notation.
For finite products we use
$$
(z;q)_n=(z)_n=
\begin{cases}
{\displaystyle\prod_{j=0}^{n-1}(1-zq^j)}, & n>0 \\
1,                         & n=0.
\end{cases}
$$
For infinite products we use
$$
(z;q)_\infty=(z)_\infty = \lim_{n\to\infty} (z;q)_n
=\prod_{n=1}^\infty (1-z q^{n-1}),
$$
$$
(z_1,z_2,\dots,z_k;q)_\infty = (z_1;q)_\infty (z_2;q)_\infty
\cdots (z_k;q)_\infty,
$$
$$
[z;q]_\infty = (z;q)_\infty (z^{-1}q;q)_\infty=
\prod_{n=1}^\infty (1-z q^{n-1})(1-z^{-1}q^{n}),
$$
$$
[z_1,z_2,\dots,z_k;q]_\infty = [z_1;q]_\infty [z_2;q]_\infty
\cdots [z_k;q]_\infty,
$$
for $\abs{q}<1$ and $z$, $z_1$, $z_2$,\dots, $z_k\ne 0$.
For $\theta$-products we use
\beqs
J_{a,b}=(q^a,q^{b-a},q^b;q^b)_\infty \quad\mbox{and}\quad
J_b=(q^b;q^b)_\infty,
\eeqs
and as usual
\beq
\eta(\tau) = \exp(\pi i\tau/12) \prod_{n=1}^\infty (1 - \exp(2\pi in\tau))
= q^{1/24} \prod_{n=1}^\infty (1- q^n),
\omylabel{eq:etadef}
\eeq
where $\IM(\tau)>0$.

Throughout this paper we let $\lfloor x \rfloor$ denote the largest
integer less than or equal to $x$, and let $\lceil x \rceil$
denote the smallest integer greater than or equal to $x$.

     We need some notation for formal Laurent series.
See the remarks at the end of \cite[Section 1, p.823]{Pa-Ra12}\omycite{Pa-Ra12}.
Let $R$ be a ring and $q$ be an indeterminant. We let $R((q))$ denote the
formal Laurent series in $q$ with coefficients in $R$. These are
series of the form
$$
f = \sum_{n\in\mathbb{Z}} a_n \, q^n,
$$
where $a_n \ne 0$ for at most finitely many $n < 0$. For $f\ne0$
we define the order of $f$ (with respect to $q$) as the smallest
integer
$N$ such that $a_N\ne0$ and write $N=\ord_q(f)$. We note that if $f$
is a modular function this coincides with $\ord(f,\infty)$.
See equation \omyeqn{ordfz} below for this other notation.
Suppose $t$ and $f\in R((q))$ and the composition $f\circ t$ is well-defined
as a formal Laurent series. This is the case if $\ord_q(t)>0$.
The $t$-order of
$$
F = f \circ t = \sum_{n\in\mathbb{Z}} a_n \, t^n,
$$
where $t = \sum_{n\in\mathbb{Z}} b_n \, q^n$, is defined to be the
smallest integer $N$ such that $a_N\ne0$ and write $N=\ord_t(F)$.
For example, if
$$
f = {q}^{-1} + 1 + 2\,q + \cdots, \qquad
t = q^2 + 3q^3 + 5q^4 + \cdots,
$$
then
$$
F = f \circ t
 = {t}^{-1} + 1 + 2\,t + \cdots, \qquad
 = q^{-2} - 3{q}^{-1} + 5 + \cdots,
$$
so that
$\ord_q(f) = -1$, $\ord_q(t)=2$, $\ord_t(F)=-1$ and $\ord_q(F) = -2$.

\section{Modular Functions}
\omylabel{sec:modfuncs}

In this section we present the needed theory of modular functions
which we use to prove identities. A general reference is Rankin's book
\cite{Ra}\omycite{Ra}.

\subsection{Background theory}
\omylabel{subsec:bthy}
Our presentation is based on \cite[pp.326-329]{Be-RNIII}\omycite{Be-RNIII}.
Let $\uhp = \{\tau\,:\,\IM(\tau)>0\}$ (the complex upper half-plane).
For each $M=\abcdMAT \in M_2^{+}(\mathbb{Z})$, the set of integer
$2\times 2$ matrix with positive determinant, the bilinear
transformation $M(\tau)$ is defined by
$$
M\tau = M(\tau) = \frac{a\tau +b}{c\tau +d}.
$$
The stroke operator is defined by
$$
\left(\stroke{f}{M}\right)(\tau) = f(M\tau),
$$
and satisfies
$$
\stroke{f}{MS} = \stroke{\stroke{f}{M}}{S}.
$$
The modular group $\Gamma(1)$ is defined by
$$
\Gamma(1) = \left\{\abcdMAT\in M_2^{+}(\mathbb{Z})\,:\, ad -bc=1\right\}.
$$
We consider the following subgroups $\Gamma$ of the modular group
with finite index
$$
\Gamma_0(N) =
\left\{\abcdMAT\in\Gamma(1)\,:\, c\equiv 0 \pmod{N}\right\},
$$
$$
\Gamma_1(N) =
\left\{\abcdMAT\in\Gamma(1)\,:\, \abcdMAT\equiv\GoneMAT\pmod{N} \right\}.
$$
Such a group $\Gamma$ acts on
$\uhp \cup \mathbb{Q} \cup \{\infty\}$ by the transformation $V(\tau)$,
for $V\in\Gamma$ which induces an equivalence relation. We
call a set $\mathscr{F}\subseteq \uhp \cup \mathbb{Q} \cup \{\infty\}$
a \textit{fundamental set} for $\Gamma$ if it contains one element of
each equivalence class. The finite set $\mathcal{F} \cap \left(\mathbb{Q}
\cup \{\infty\}\right)$ is called the \textit{complete set of inequivalent cusps}.

A function $f\,:\,\mathscr{H} \longrightarrow \mathbb{C}$ is
a \textit{weakly holomorphic modular function} on $\Gamma$ if the following conditions
hold:
\begin{enumerate}
\item[(i)] $f$ is holomorphic on $\uhp$.
\item[(ii)] $\displaystyle \stroke{f}{V} = f$ for all $V\in\Gamma$.
\item[(iii)] For every $A\in\Gamma(1)$ the function $\stroke{f}{A^{-1}}$
has an expansion
$$
(\stroke{f}{A^{-1}})(\tau) = \sum_{m=m_0}^\infty b(m) \exp(2\pi i\tau m/\kappa)
$$
on some half-plane $\left\{\tau\,:\,\IM \tau > h \ge 0\right\}$,
where $T=\TMAT$ and
$$
\kappa = \min \left\{k>0\,:\, \pm A^{-1} T^k A \in \Gamma\right\}.
$$
\end{enumerate}
The positive integer $\kappa = \kappa(\Gamma;\zeta)$
is called the \textit{fan width} of
$\Gamma$ at the cusp $\zeta = A^{-1}\infty$.
If $b(m_0)\ne 0$, then we write
$$
\ORD(f,\zeta,\Gamma) = m_0
$$
which is called the \textit{order} of $f$ at $\zeta$ with respect to
$\Gamma$. We also write
\beq
\ord(f;\zeta) = \frac{m_0}{\kappa} = \frac{m_0}{\kappa(\Gamma,\zeta)},
\omylabel{eq:ordfz}
\eeq
which is called the \textit{invariant order} of $f$ at $\zeta$.
For each $z\in\uhp$, $\ord(f;z)$ denotes the order of
$f$ at $z$ as an analytic function of $z$, and the order of $f$ with
respect to $\Gamma$ is defined by
$$
\ORD(f,z,\Gamma) = \frac{1}{\ell} \ord(f;z)
$$
where $\ell$ is the order of $z$ as a fixed point of $\Gamma$.
We note $\ell =1$, $2$ or $3$.
Our main tool for proving modular function identities
is the valence formula \cite[Theorem 4.1.4, p.98]{Ra}\omycite{Ra}.
If $f\ne0$ is a modular function on $\Gamma$ and $\mathscr{F}$ is any
fundamental set for $\Gamma$ then
\beq
\sum_{z\in\mathscr{F}} \ORD(f,z,\Gamma) = 0.
\omylabel{eq:valform}
\eeq

\subsection{Eta-product identities}
\omylabel{subsec:etaprods}
We will consider eta-products of the form
\begin{equation}
f(\tau) = \prod_{d\mid N} \eta(d\tau)^{m_{d}},
\omylabel{eq:etapdef}
\end{equation}
where $N$ is a positive integer,  each $d>0$ and $m_{d}\in\Z$.

\subsubsection*{Modularity}
Newman \cite{Ne59}\omycite{Ne59} has found necessary and sufficient conditions
under which an eta-product is a modular function on $\Gamma_0(N)$.
\begin{theorem}[{\cite[Theorem 4.7]{Ne59}\omycite{Ne59}}]
\omylabel{thm:etamodthm}
The function $f(\tau)$ (given in \omyeqn{etapdef}) is a modular function
on $\Gamma_0(N)$ if and only if
\begin{enumerate}
\item
$\displaystyle\sum_{d\mid N} m_d = 0$,
\item
$\displaystyle\sum_{d\mid N} d m_d \equiv0\pmod{24}$,
\item
$\displaystyle\sum_{d\mid N} \frac{N m_d}{d} \equiv0\pmod{24}$, and
\item
$\displaystyle\prod_{d\mid N} d^{|m_d|}$ is a square.
\end{enumerate}
\end{theorem}

\subsubsection*{Orders at cusps}
Ligozat \cite{Li75}\omycite{Li75} has computed the invariant order of an eta-product
 at the cusps of $\Gamma_0(N)$.
\begin{theorem}[{\cite[Theorem 4.8]{Li75}\omycite{Li75}}]
\omylabel{thm:ordthm}
If the eta-product $f(\tau)$ (given in \omyeqn{etapdef})   is a modular function
on $\Gamma_0(N)$, then its order at the cusp $\zeta=\frac{b}{c}$
(assuming $(b,c)=1$) is
\begin{equation}
\ord(f(\tau);\zeta)=\sum_{d\mid N} \frac{(d,c)^2 m_d}{24d}.
\mylabel{eq:ecord}
\end{equation}
\end{theorem}

Chua and Lang \cite{Ch-La04}\omycite{Ch-La04}
have found a set of inequivalent
cusps for $\Gamma_0(N)$.
\begin{theorem}[{\cite[p.354]{Ch-La04}\omycite{Ch-La04}}]
\omylabel{thm:chualang}
Let N be a positive integer and for each positive divisor $d$ of $N$ let
$e_d = (d,N/d)$. Then the set
\beqs
\Delta = {\underset{d\mid N}{\cup}} \, S_d
\eeqs
is a complete set of inequivalent cusps of $\Gamma_0(N)$ where
$$
S_d = \{ x_i/d\,:\,(x_i,d)=1,\quad 0\le x_i\le d-1,\quad x_i\not\equiv
x_j \pmod{e_d}\}.
$$
\end{theorem}
Biagioli \cite{Bi89}\omycite{Bi89} has found the fan width of the cusps of
$\Gamma_0(N)$.
\begin{lemma}[{\cite[Lemma 4.2]{Bi89}\omycite{Bi89}}]
\omylabel{lem:fanw}
If $(r,s)=1$, then the fan width of $\Gamma_0(N)$ at $\frac{r}{s}$
is
$$
\kappa\left(\Gamma_0(N); \frac{r}{s}\right) = \frac{N}{(N,s^2)}.
$$
\end{lemma}

\subsubsection*{An application of the valence formula}
Since eta-products have no zeros or poles in $\uhp$ the following result
follows easily from the valence formula \omyeqn{valform}.
\begin{theorem}
\omylabel{thm:valcor}
Let $f_1(\tau)$, $f_2(\tau)$, \dots, $f_n(\tau)$ be eta-products that
are modular functions on $\Gamma_0(N)$. Let $\mathcal{S}_N$ be a set of inequivalent
cusps for $\Gamma_0(N)$. Define the constant
\beq
B = \sum_{\substack{\zeta\in\mathcal{S}_N\\ \zeta\ne \infty}}
        \mbox{min}
        (\left\{\ORD(f_j,\zeta,\Gamma_0(N))\,:\, 1 \le j \le n\right\}),
\omylabel{eq:Bdef}
\eeq
and consider
\beq
g(\tau) := \alpha_1 f_1(\tau) + \alpha_2 f_2(\tau) + \cdots + \alpha_n f_n(\tau),
\omylabel{eq:gdef}
\eeq
where each $\alpha_j\in\mathbb{C}$. Then
$$
g(\tau) \equiv 0
$$
if and only if
\beq
\ORD(g(\tau), \infty, \Gamma_0(N)) > -B.
\omylabel{eq:ORDBineq}
\eeq
\end{theorem}

\noindent
\textit{An algorithm for proving eta-product identities.}

        \vskip 10pt\noindent
{\it\footnotesize STEP 0}. \quad  Write the identity in the following
form:
\begin{equation}
    \alpha_1 f_1(\tau) + \alpha_2 f_2(\tau) + \cdots + \alpha_n f_n(\tau)  = 0,
\omylabel{eq:fid}
\end{equation}
where each $\alpha_i\in\C$ and each $f_i(\tau)$ is an eta-product of
level $N$.

        \vskip 10pt\noindent
{\it\footnotesize STEP 1}. \quad  Use Theorem \omythm{etamodthm} to check that
$f_j(\tau)$ is a modular function on $\Gamma_0(N)$ for each
$1 \le j \le n$.

        \vskip 10pt\noindent
{\it\footnotesize STEP 2}. \quad  Use Theorem \omythm{chualang} to
find a set $\mathcal{S}_N$ of inequivalent cusps for $\Gamma_0(N)$ and the
fan width of each cusp.

        \vskip 10pt\noindent
{\it\footnotesize STEP 3}. \quad  Use Theorem \omythm{ordthm} to
calculate the order of each eta-product
$f_j(\tau)$ at each cusp of $\Gamma_0(N)$.

        \vskip 10pt\noindent
{\it\footnotesize STEP 4}. \quad  Calculate
        $$
        B =
        \sum_{\substack{\zeta\in\mathcal{S}_N\\ \zeta\ne \infty}}
        \mbox{min}
        (\left\{\ORD(f_j,\zeta,\Gamma_0(N))\,:\, 1 \le j \le n\right\} ).
        $$

        \vskip 10pt\noindent
{\it\footnotesize STEP 5}. \quad  Show that
        $$
        \ORD(g(\tau),\infty,\Gamma_0(N)) > -B
        $$
        where
        $$
        g(\tau) = \alpha_1 f_1(\tau) + \alpha_2 f_2(\tau) +
        \cdots + \alpha_n f_n(\tau).
        $$
        Theorem \omythm{valcor} then implies that $g(\tau)\equiv0$ and
        hence the eta-product identity  \omyeqn{fid}.

The third author has written a \textsc{MAPLE} package
called \texttt{ETA} which implements this algorithm. See
\begin{center}
\url{http://qseries.org/fgarvan/qmaple/ETA/}
\end{center}

\subsubsection*{A modular equation}
Define
\begin{align}
\mylabel{eq:t7def}
t:=t(\tau):=\frac{\eta(7\tau)^4}{\eta(\tau)^4}.
\end{align}
We note that $t(\tau)$ is a Hauptmodul for $\Gamma_0(7)$
\cite{Ma09}\omycite{Ma09}.
As an application of our algorithm we prove the following theorem which will be needed
later.
\begin{theorem}
\omylabel{thm:modeq}
Let
\begin{align}
a_0(t)&=t,\omylabel{eq:a0}\\
a_1(t)&=7^2t^2+4\cdot 7t,\omylabel{eq:a1}\\
a_2(t)&=7^4t^3+4\cdot 7^3t^2+46\cdot 7t,\omylabel{eq:a2}\\
a_3(t)&=7^6t^4+4\cdot 7^5t^3+46\cdot 7^3t^2+272\cdot 7t,\omylabel{eq:a3}\\
a_4(t)&=7^8t^5+4\cdot 7^7t^4+46\cdot 7^5t^3+272\cdot 7^3t^2+845\cdot 7t,\omylabel{eq:a4}\\
a_5(t)&=7^{10}t^6+4\cdot 7^9t^5+46\cdot 7^7t^4+272\cdot 7^5t^3+845\cdot 7^3t^2+176\cdot 7^2t,\omylabel{eq:a5}\\
a_6(t)&=7^{12}t^7+4\cdot 7^{11}t^6+46\cdot 7^9t^5+272\cdot 7^7t^4+845\cdot 7^5t^3+176\cdot 7^4t^2+82\cdot 7^2t.\omylabel{eq:a6}
\end{align}
where $t=t(\tau)$ is defined in \eqn{t7def}.
Then
\beq
t(\tau)^7-\sum_{l=0}^{6}a_l(t(7\tau))t(\tau)^l=0.
\omylabel{eq:modeq}
\eeq
\end{theorem}
\begin{proof}
From Theorem \omythm{etamodthm} we find that $t(\tau)$ is a modular
function on $\Gamma_0(7)$ and $t(7\tau)$ is a modular function
on $\Gamma_0(49)$. Hence each term on the left side of \omyeqn{modeq}
is a modular function on $\Gamma_0(49)$. For convenience
we divide by $t(\tau)^7$ and let
\beq
g(\tau) = 1 - \sum_{l=0}^{6}a_l(t(7\tau))t(\tau)^{l-7}.
\omylabel{eq:gsum}
\eeq
From Theorem \omythm{chualang}, Lemma \omylem{fanw} and Theorem \omythm{ordthm}
we have the
following table of fan widths for the cusps of $\Gamma_0(49)$, with
the orders and invariant orders of both $t(\tau)$ and $t(7\tau)$.
$$ 
\begin{array}{|c|c|c|c|c|c|c|c|c|} 
\noalign{\hrule} 
\zeta &0& 1/7& 2/7& 3/7& 4/7& 5/7& 6/7& 1/49 \\
\noalign{\hrule} 
\kappa(\Gamma_0(49),\zeta)& 49& 1& 1& 1& 1& 1 &1&1 \\ 
\noalign{\hrule} 
\ord(t(\tau),\zeta)&      -1/7& 1& 1& 1& 1& 1& 1& 1 \\
\noalign{\hrule} 
\ORD(t(\tau),\zeta,\Gamma_0(49))&-7& 1& 1& 1& 1& 1& 1& 1 \\
\noalign{\hrule} 
\ord(t(7\tau),\zeta)&-1/49& -1& -1& -1& -1& -1& -1& 7 \\
\noalign{\hrule} 
\ORD(t(7\tau),\zeta,\Gamma_0(49))&-1& -1& -1& -1& -1& -1& -1& 7 \\
\noalign{\hrule}
\end{array}
$$

Expanding the right side of \omyeqn{gsum} gives $29$ terms of the
form $t(7\tau)^k t(\tau)^{j-7}$ with $1\le k
\le j+1$ where $0 \le j \le 6$, together with $(k,j)=(0,7)$.
We calculate the order of each term at each cusp $\zeta$ of $\Gamma_0{(49)}$,
and thus giving  lower bounds for $\ORD(g(\tau),\zeta, \Gamma_0(49))$.
$$
\begin{array}{|c|c|c|c|c|c|c|c|c|} 
\noalign{\hrule} 
\zeta &0& 1/7& 2/7& 3/7& 4/7& 5/7& 6/7& 1/49 \\
\noalign{\hrule} 
\ORD(g(\tau),\zeta,\Gamma_0(49))\ge & 0& -8& -8& -8& -8& -8& -8& 0\\
\noalign{\hrule}
\end{array}
$$

Thus the constant $B$ in Theorem \omythm{valcor} is $B=-48$. It suffices to show
that
$$
\ORD(g(\tau),\infty,\Gamma_0(49))> 48.
$$
This is easily verified. Thus by Theorem \omythm{valcor} we have $g(\tau) \equiv 0$ and
the result follows.
\end{proof}

\subsection{The $U_p$ operator}
\omylabel{subsec:Upop}

Let $p$ prime and
$$
f = \sum_{m=m_0}^\infty a(m) q^m
$$
be a formal Laurent series. We
define $U_p$ by
\beq
U_p(f) := \sum_{p m \ge m_0}  a(p m) q^m.
\omylabel{eq:Updeffls}
\eeq
If $f$ and $h$ are  modular functions (with $q=\exp(2\pi i\tau)$),
\beq
U_p(f) = \frac{1}{p} \sum_{j=0}^p \stroke{f}{\MAT{1/p}{j/p}{0}{1}}
 = \frac{1}{p} \sum_{j=0}^p f\left(\frac{\tau +j}{p}\right),
\omylabel{eq:Updef}
\eeq

and for
\begin{align*}
H(\tau)=h(p\tau),
\end{align*}
we have
\begin{align}
U_p(fH)(\tau)=h(\tau)U_p(f)(\tau).
\mylabel{eq:u71}
\end{align}

\begin{theorem}[{\cite[Lemma 7, p.138]{At-Le70}\omycite{At-Le70}}]
\omylabel{thm:ALUpthm}
Let $p$ be prime. If $f$ is a modular function on $\Gamma_0(pN)$
and $p\mid N$, then $U_p(f)$ is a modular function
on $\Gamma_0(N)$.
\end{theorem}


\section{The rank parity function modulo powers of $7$}
\omylabel{sec:rankparity5}
\subsection{A Generating Function}
\omylabel{subsec:genfunc}

In this section we prove an identity for the generating function of
$$
a_f(n/7) - a_f(7n-2),
$$
where it is understood that $a_f(n)=0$ if $n$ is not a non-negative integer.

\begin{theorem}We have
\omylabel{thm:af7thm}
\begin{align}
f(q^7)-q^2f_5(q)=\frac{J_7^3}{J_2^2}\left(\frac{J_1^3J_7^3}{J_2^3J_{14}^3}+6q^2\frac{J_{14}^4J_1^4}{J_2^4J_7^4}\right).
\mylabel{eq:af7id}
\end{align}
\end{theorem}
\begin{remark}
We note that this theorem can also be proved from Theorem \mythm{f0Up}.
\end{remark}

\begin{proof}
From Watson \cite[p.64]{Wa36a}\omycite{Wa36a} we have
\begin{align}
f(q)&=\frac{2}{(q;q)_\infty}\sum_{n=-\infty}^{\infty}\frac{(-1)^nq^{n(3n+1)/2}}{1+q^n}.
\omylabel{eq:fqid}
\end{align}

We find that
\begin{align*}
\sum_{n=-\infty}^{\infty}\frac{(-1)^nq^{n(3n+1)/2+4n}}{1+q^{7n}}&
=\sum_{n=-\infty}^{\infty}\frac{(-1)^nq^{n(3n+1)/2+2n}}{1+q^{7n}},\\
\sum_{n=-\infty}^{\infty}\frac{(-1)^nq^{n(3n+1)/2+5n}}{1+q^{7n}}&
=\sum_{n=-\infty}^{\infty}\frac{(-1)^nq^{n(3n+1)/2+n}}{1+q^{7n}},\\
\sum_{n=-\infty}^{\infty}\frac{(-1)^nq^{n(3n+1)/2+6n}}{1+q^{7n}}&
=\sum_{n=-\infty}^{\infty}\frac{(-1)^nq^{n(3n+1)/2}}{1+q^{7n}}.
\nonumber
\end{align*}
By \cite[Theorem 2.1]{Ch05}\omycite{Ch05} we have
\begin{align}
\sum_{n=-\infty}^{\infty}\frac{(-1)^nq^{n(3n+1)/2}}{1+q^{7n}}
=&-P(q^7,-q^7;q^{49})+q^{-6}P(q^{14},-q^7;q^{49})-q^{-9}P(q^{21},-q^7;q^{49})\mylabel{eq:f-cong7-0}\\
&+\frac{J_1}{J_{49}}\sum_{n=-\infty}^{\infty}\frac{(-1)^nq^{(147n^2+49n)/2-7}}{1+q^{49n-7}},\nonumber\\
\sum_{n=-\infty}^{\infty}\frac{(-1)^nq^{n(3n+1)/2+n}}{1+q^{7n}}
=&P(q^{21},-q^{21};q^{49})-q^{3}P(q^{14},-q^{28};q^{49})+q^{9}P(q^{7},-q^{28};q^{49})\mylabel{eq:f-cong7-1}\\
&-\frac{J_1}{J_{49}}\sum_{n=-\infty}^{\infty}\frac{(-1)^nq^{(147n^2-7n)/2-7}}{1+q^{49n-14}},\nonumber\\
\sum_{n=-\infty}^{\infty}\frac{(-1)^nq^{n(3n+1)/2+2n}}{1+q^{7n}}
=&P(q^{14},-q^{14};q^{49})-q^{6}P(q^{7},-q^{14};q^{49})-q^{-3}P(q^{21},-q^{14};q^{49})\mylabel{eq:f-cong7-2}\\
&-\frac{J_1}{J_{49}}\sum_{n=-\infty}^{\infty}\frac{(-1)^nq^{(147n^2+147n)/2+13}}{1+q^{49n+14}},\nonumber\\
\sum_{n=-\infty}^{\infty}\frac{(-1)^nq^{n(3n+1)/2+3n}}{1+q^{7n}}
=&q^{-9}P(q^{7},-1;q^{49})-q^{-15}P(q^{14},-1;q^{49})+q^{-18}P(q^{21},-1;q^{49})\mylabel{eq:f-cong7-3}\\
&-\frac{J_1}{J_{49}}\sum_{n=-\infty}^{\infty}\frac{(-1)^nq^{(147n^2+49n)/2-2}}{1+q^{49n}}\nonumber,
\end{align}
where
\beq
P(a,b;q)=\frac{[a,a^2;q]_\infty (q;q)_\infty^2}{[b/a,ab,b;q]_\infty}.
\omylabel{eq:Pabqdef}
\eeq

From \omyeqn{fqid}-\omyeqn{f-cong7-3}, and noting that 
$P(q^7,-q^7;q^{49})=P(q^{14},-q^{14};q^{49})$
we have
\begin{align}
f(q)
=&\frac{2}{(q;q)_\infty}\sum_{n=-\infty}^{\infty}\frac{(-1)^nq^{n(3n+1)/2}}{1+q^n}\mylabel{eq:f-7-dissection}\\
=&\frac{2}{(q;q)_\infty}\sum_{n=-\infty}^{\infty}
  \frac{(-1)^nq^{n(3n+1)/2}(1-q^n+q^{2n}-q^{3n}+q^{4n}-q^{5n}+q^{6n})}{1+q^{7n}}\nonumber\\
=&\frac{2}{(q;q)_\infty}\sum_{n=-\infty}^{\infty}\frac{(-1)^nq^{n(3n+1)/2}(2-2q^n+2q^{2n}-q^{3n})}{1+q^{7n}}\nonumber\\
=&\frac{2}{J_1}(2q^{-6}P(q^{14},-q^7;q^{49})-2q^{-9}P(q^{21},-q^7;q^{49})+2q^{3}P(q^{14},-q^{28};q^{49})\nonumber\\
&-2q^{9}P(q^{7},-q^{28};q^{49})-2q^{6}P(q^{7},-q^{14};q^{49})-2q^{-3}P(q^{21},-q^{14};q^{49})\nonumber\\
&-q^{-9}P(q^{7},-1;q^{49})+q^{-15}P(q^{14},-1;q^{49})-q^{-18}P(q^{21},-1;q^{49})-\frac{J_7^4}{J_{14}^2})\nonumber\\
&+\frac{4}{J_{49}}\sum_{n=-\infty}^{\infty}\frac{(-1)^nq^{(147n^2+49n)/2-7}}{1+q^{49n-7}}
+\frac{4}{J_{49}}\sum_{n=-\infty}^{\infty}\frac{(-1)^nq^{(147n^2-7n)/2-7}}{1+q^{49n-14}}\nonumber\\
&-\frac{4}{J_{49}}\sum_{n=-\infty}^{\infty}\frac{(-1)^nq^{(147n^2+147n)/2+13}}{1+q^{49n+14}}
+\frac{1}{q^2}f(q^{49})\nonumber.
\end{align}

We let
\begin{align}
\mylabel{eq:gq-defn}
g(q):=&\frac{2}{J_1}(2q^{-6}P(q^{14},-q^7;q^{49})-2q^{-9}P(q^{21},-q^7;q^{49})+2q^{3}P(q^{14},-q^{28};q^{49})\\
&-2q^{9}P(q^{7},-q^{28};q^{49})-2q^{6}P(q^{7},-q^{14};q^{49})-2q^{-3}P(q^{21},-q^{14};q^{49})
\nonumber\\
&-q^{-9}P(q^{7},-1;q^{49})+q^{-15}P(q^{14},-1;q^{49})-q^{-18}P(q^{21},-1;q^{49})-\frac{J_7^4}{J_{14}^2}),
\nonumber
\end{align}
write the $7$-dissection of $g(q)$ as
\begin{align}
g(q)=g_0(q^7) +q\,g_1(q^7)+\cdots+q^6\,g_6(q^7).
\omylabel{eq:gq7}
\end{align}

From \omyeqn{fqid}, \omyeqn{f-7-dissection} and \omyeqn{gq7}, replacing $q^7$ by $q$, we have
\begin{align}\mylabel{eq:compare-f-7}
\sum_{n=0}^{\infty}a_f(7n+5)q^n=\frac{1}{q^2}f(q^7)+g_5(q)
\end{align}
after dividing both sides by $q^5$ and replacing $q^7$ by $q$.

The 7-dissection of $J_1$ is well-known
\begin{align}\mylabel{eq:q-expand-49}
J_1=J_{49}\times\left(A(q^7)-q-B(q^7)q^2+\frac{q^5}{A(q^7)B(q^7)}\right),
\end{align}
where
$$
A(q):=\frac{J_{2,7}}{J_{1,7}}, \quad
B(q):=\frac{J_{3,7}}{J_{2,7}}.
$$
See for example \cite[Lemma 3.18]{Ga88b}\omycite{Ga88b}.

From \omyeqn{gq-defn}, \omyeqn{compare-f-7} and \omyeqn{q-expand-49},
\begin{align}
&J_{49}(g_0(q^7)+qg_1(q^7)+\cdots+q^6g_6(q^7))\left(A(q^7)-qB(q^{7})-q^2+\frac{q^5}{A(q^7)B(q^{7})}\right)\mylabel{eq:g-expand}\\
=&2(2q^{-6}P(q^{14},-q^7;q^{49})-2q^{-9}P(q^{21},-q^7;q^{49})+2q^{3}P(q^{14},-q^{28};q^{49})\nonumber\\
&-2q^{9}P(q^{7},-q^{28};q^{49})-2q^{6}P(q^{7},-q^{14};q^{49})-2q^{-3}P(q^{21},-q^{14};q^{49})\nonumber\\
&-q^{-9}P(q^{7},-1;q^{49})+q^{-15}P(q^{14},-1;q^{49})-q^{-18}P(q^{21},-1;q^{49})-\frac{J_7^4}{J_{14}^2})\nonumber.
\end{align}

By expanding the left side of \omyeqn{g-expand} and comparing both sides
according to the residue of the exponent of $q$ modulo 7,
we obtain 7 equations:
\begin{align}
\mylabel{eq:7-eq0}
&A(q^7)g_0+\frac{q^7g_2}{A(q^7)B(q^{7})}-q^7g_5-q^7B(q^{7})g_6=\frac{2J_7^4}{J_{14}^2J_{49}},\\
&-B(q^7)g_0+A(q^7)g_1+\frac{q^7g_3}{A(q^7)B(q^7)}-q^7g_6=\frac{4P(q^{14},-q^7)}{q^7J_{49}},\\
&g_0+B(q^7)g_1-A(q^7)g_2-\frac{q^7g_4}{A(q^7)B(q^7)}=\frac{4q^7P(q^7,-q^{21})}{J_{49}},\\
&g_1+B(q^7)g_2-A(q^7)g_3-\frac{q^7g_5}{A(q^7)B(q^7)}=\frac{2P(q^{21},-1)-4q^{21}P(q^{14},-q^{21})}{q^{21}J_{49}},\\
&g_2+B(q^7)g_3-A(q^7)g_4-\frac{q^7g_6}{A(q^7)B(q^7)}=\frac{4P(q^{21},-q^{14})}{q^7J_{49}},\\
&\frac{-g_0}{A(q^7)B(q^7)}+g_3+B(q^7)g_4-A(q^7)g_5=\frac{2P(q^7,-1)+4P(q^{21},-q^7)}{q^{14}J_{49}},\\
\mylabel{eq:7-eq6}
&\frac{g_1}{A(q^7)B(q^7)}-g_4-B(q^7)g_5+A(q^7)g_6=\frac{2P(q^{14},-1)-4q^{21}P(q^{7},-q^{14})}{q^{21}J_{49}}.
\end{align}
where $g_j=g_j(q^7)$ for $0\le j \le 6$.

Solving these equations we find that
\begin{align*}
g_5(q)=&\frac{1}{H}\big\{(A^6B^9+4A^7B^7+3A^8B^5-A^9B^3+3A^4B^6q+8A^5B^4q\\
&-4A^6B^2q-4A^2B^3q^2-3A^3Bq^2+q^3)A^2B^3X_0\\
&+(A^6B^9+3A^7B^7+A^8B^5-2A^3B^8q-3A^4B^6q+6A^5B^4q\\
&+AB^5q^2+2A^2B^3q^2-q^3)A^3B^2X_1\\
&+(A^7B^7+2A^8B^5+A^3B^8q+2A^4B^6q+3A^5B^4q+A^6B^2q\\
&-6A^2B^3q^2+3A^3Bq^2+q^3)A^3B^3X_2\\
&+(A^{10}B^8+A^{11}B^6+A^6B^9q+4A^8B^5q+6A^4B^6q^2\\
&-A^5B^4q^2-5A^2B^3q^3+q^4)ABX_3\\
&+(A^9B^5+A^4B^8q+6A^5B^6q+2A^6B^4q-3A^2B^5q^2+3A^3B^3q^2\\
&-A^4Bq^2+B^2q^3-2Aq^3)A^3B^3X_4\\
&+(A^9B^3-A^3B^8q-4A^4B^6q+A^5B^4q+5A^6B^2q\\
&+AB^5q^2+6A^3Bq^2-q^3)A^4B^4X_5\\
&+(A^5B^{11}+5A^6B^9+6A^7B^7-A^8B^5-A^4B^6q\\
&-4A^2B^3q^2+A^3Bq^2+q^3)qA^2B^2X_6\big\},
\end{align*}
where $A:=A(q)$, $B:=B(q)$,

\begin{align*}
H:=&-A^7B^{14}q+A^{14}B^7-7A^8B^{12}q-14A^9B^{10}q+7A^{11}B^6q-8A^7B^7q^2\\
\nonumber
&+14A^8B^5q^2+14A^4B^6q^3-7A^2B^3q^4+q^5,
\end{align*}
and $X_0$ -- $X_6$ are the right sides of 
\omyeqn{7-eq0} -- \omyeqn{7-eq6} (respectively) after replacing $q^7$ by $q$.
Then using the third author's \texttt{thetaids} \textsc{MAPLE} package, see
\begin{center}
\url{http://qseries.org/fgarvan/qmaple/thetaids/}
\end{center}
we can prove,
\begin{align}
\label{remark}
H=\frac{J_1^8}{J_7^8}A^7B^7,
\end{align}
and then\begin{align}
\mylabel{eq:f5g}
g_5(q)=-\frac{J_7^3}{J_2^2}\left(\frac{J_1^3J_7^3}{q^2J_2^3J_{14}^3}+6\frac{J_1^4J_{14}^4}{J_2^4J_7^4}\right).
\end{align}

From \omyeqn{compare-f-7} and \omyeqn{f5g} we have
$$
f(q^7)-\sum_{n=0}^{\infty}a_f(7n-2)q^n=q^2g_5(q)=\frac{J_7^3}{J_2^2}\left(\frac{J_1^3J_7^3}{J_2^3J_{14}^3}+6q^2\frac{J_1^4J_{14}^4}{J_2^4J_7^4}\right),
$$
which is our result \eqref{eq:af7id}.
\end{proof}
\subsection{A Fundamental Lemma}
\omylabel{subsec:fundlem5}
We need the following fundamental lemma,
whose proof follows easily from Theorem \thm{modeq}.
\begin{lemma}[A Fundamental Lemma]
\omylabel{lem:fun7}
Suppose $u=u(\tau)$, and $j$
is any integer. Then
\begin{align*}
{U_7}(u\,t^j) =\sum_{l=0}^{6}a_l(\tau)\,{U_7}(u\,t^{j+l-7}),
\end{align*}
where $t=t(\tau)$ is defined in \omyeqn{t7def} and the $a_j(\tau)$ are given
in \omyeqn{a0}--\omyeqn{a6}.
\end{lemma}
\begin{proof}
The result follows easily from \omyeqn{modeq} by multiplying both
sides by $u\,t^{j-7}$ and applying $U_7$.
\end{proof}

We can check for each $a_j(t)$ that there exist integers $s(j,l)$ satisfying
\begin{align}
a_j(t)=\sum_{l=1}^{7}s(j,l)7^{[(7l+j-4)/4]}t^l.
\mylabel{eq:aj}
\end{align}

Let $g=\sum_{n}a_nt^n,g\neq 0$, be such that $a_n=0$ for almost all $n<0$. Then the order of $g$ is the smallest integer $N$ such that $a_N\neq 0$, and we write $N=ord_t(g)$.

\begin{lemma}\label{lem1}
Let $u,v_1,v_2,v_3:\mathbb{H}\rightarrow \mathbb{C}$ and $l\in \mathbb{Z}$. Suppose for $l\leq k\leq l+6$ and $i=1,2,3$ there exist Laurent polynomials $p_k^{(i)}(t)\in \mathbb{Z}[t,t^{-1}]$ such that
\begin{align}
U_7(ut^k)=v_1p_k^{(1)}(t)+v_2p_k^{(2)}(t)+v_3p_k^{(3)}(t),
\mylabel{eq:lem11}
\end{align}
and
\begin{align}
ord_t(p_k^{(i)}(t))\geq \left[\frac{k+s_i}{7}\right],
\mylabel{eq:lem12}
\end{align}
for fixed integers $s_i$. Then there exist families of Laurent polynomials $p_k^{(i)}(t)\in \mathbb{Z}[t,t^{-1}]$, $k\in \mathbb{Z}$, such that \omyeqn{lem11} and \omyeqn{lem12} hold for all $k\in \mathbb{Z}$.
\end{lemma}

\begin{proof}
Let $N>l+6$ be an integer and assume by induction that there are families of Laurent polynomials $p_k^{(i)}(t)$, $i\in {1,2,3}$, such that \omyeqn{lem11} and \omyeqn{lem12} hold for $l\leq k\leq N-1$. Suppose
\begin{align*}
p_k^{(i)}(t)=\sum_{n\geq [(k+s_i)/7]}c_i(k,n)t^n,\text{ }1\leq k\leq N-1,
\end{align*}
with integers $c_i(k,n)$. Applying Lemma \ref{lem:fun7} we obtain:
\begin{align*}
U_7(ut^N)&=\sum_{j=0}^{6}a_j(t)U_7(ut^{N+j-7})\\
&=\sum_{j=0}^{6}a_j(t)\sum_{i=1}^{3}v_i\sum_{n\geq [(N+j-7+s_i)/7]}c_i(N+j-7,n)t^n\\
&=\sum_{i=1}^{3}v_i\sum_{j=0}^{6}a_j(t)t^{-1}\sum_{n\geq [(N+j+s_i)/7]}c_i(N+j-7,n-1)t^n.
\end{align*}
Recalling the fact that $a_j(t)t^{-1}$ for $0\leq j\leq 6$ is a polynomial 
of $t$, this determines Laurent polynomials $P_N^{(i)}(t)$ with the desired 
properties. The induction proof for $N<l$ is analogous.    
\end{proof}

\begin{lemma}\label{lem2}
Let $u,v_1,v_2,v_3:\mathbb{H}\rightarrow \mathbb{C}$ and $l\in \mathbb{Z}$. Suppose for $l\leq k\leq l+6$ and $i=1,2,3$ there exist Laurent polynomials $p_k^{(i)}(t)\in \mathbb{Z}[t,t^{-1}]$ such that
\begin{align}
U_7(ut^k)=v_1p_k^{(1)}(t)+v_2p_k^{(2)}(t)+v_3p_k^{(3)}(t),
\mylabel{eq:lem21}
\end{align}
where
\begin{align}
p_k^{(i)}(t)=\sum_{n}c_i(k,n)7^{[\frac{7n-k+r_i}{4}]}t^n,
\mylabel{eq:lem22}
\end{align}
with integers $r_i$ and $c_i(k,n)$. Then there exist families of Laurent polynomials $p_k^{(i)}(t)\in \mathbb{Z}[t,t^{-1}]$, $k\in \mathbb{Z}$, of the form \omyeqn{lem22} for which property \omyeqn{lem21} holds for all $k\in \mathbb{Z}$.
\end{lemma}

\begin{proof}
Suppose for an integer $N>l+6$ there are families of Laurent polynomials $p_k^{(i)}(t)$, $i\in {1,2,3}$, of the form \omyeqn{lem22} satisfying property \omyeqn{lem21} for $l\leq k\leq N-1$. We proceed by mathematical induction on $N$. Applying Lemma \ref{lem:fun7} and using the induction base \omyeqn{lem21} and \omyeqn{lem22} we obtain:
\begin{align*}
U_7(ut^N)=\sum_{j=0}^{6}a_j(t)\sum_{i=1}^{3}v_i\sum_{n}c_i(N+j-7,n)7^{[\frac{7n-(N+j-7)+r_i}{4}]}t^n.
\end{align*}
Utilizing \omyeqn{aj} Lemma \ref{lem:fun7} we have                   
\begin{align*}
U_7(ut^N)&=\sum_{j=0}^{6}\sum_{l=1}^{7}s(j,l)7^{[\frac{7l+j-4}{4}]}t^l\sum_{i=1}^{3}v_i\sum_{n}c_i(N+j-7,n)7^{[\frac{7n-(N+j-7)+r_i}{4}]}t^n\\
&=\sum_{i=1}^{3}v_i\sum_{j=0}^{6}\sum_{l=1}^{7}\sum_{n}s(j,l)c_i(N+j-7,n-l)7^{[\frac{7(n-l)-(N+j-7)+r_i}{4}]+[\frac{7l+j-4}{4}]}t^n.
\end{align*}
The induction step is completed by simplifying the exponent of $7$ as follows:
\begin{align*}
&\left[\frac{7(n-l)-(N+j-7)+r_i}{4}]+[\frac{7l+j-4}{4}\right]\\
\geq &\left[\frac{7(n-l)-(N+j-7)+r_i+7l+j-4-3}{4}\right]\\
=&\left[\frac{7n-N+r_i}{4}\right].
\end{align*}
The induction proof for $N<l$ is analogous.          
\end{proof}

\subsection{Proof of  Theorem \thm{mainthm} }
The proof depends  on the forty-two fundamental relations listed in the 
Appendix \ref{funr-7}. These identities can be
proved using the algorithm described in 
\cite[Section 2C, pp.8-9]{Ch-Ch-Ga20}.
From Theorem \ref{thm:af7thm} we have
\begin{align}
\mylabel{eq:gen7}
\sum_{n=0}^{\infty}(a_f(n/7)-a_f(7n-2))q^n
=\frac{J_7^3}{J_2^2}\left(\frac{J_1^3J_7^3}{J_2^3J_{14}^3}
  +6q^2\frac{J_{14}^4J_1^4}{J_2^4J_7^4}\right).
\end{align}
For $f:\mathbb{H}\rightarrow \mathbb{C}$ we define $U_A(f)$ and 
$U_B(f):\mathbb{H}\rightarrow \mathbb{C}$ by
$$
U_A(f):=U_7(Af),\quad U_B(f):=U_7(Bf),
$$
where
$$
A:=\frac{q^8J_{98}^2}{J_2^2},\quad B:=\frac{J_1^3}{q^6J_{49}^3}.
$$
Define
$$
L_0:=7p_0+p_1,
$$ and for $\alpha\ge0$ define     
$$
L_{2\alpha+1}=U_A(L_{2\alpha}), \quad L_{2\alpha+2}=U_B(L_{2\alpha+1}),
$$
where
$$
p_0:=\frac{qJ_{14}^4J_1^4}{J_7^4J_2^4}, \quad 
p_1:=\frac{J_1^3J_7^3}{qJ_2^3J_{14}^3}-p_0.
$$
Using \omyeqn{Updeffls}, \omyeqn{u71} and \omyeqn{gen7}, it is easy to 
verify that for $\alpha\ge0$              we have
\begin{align*}
L_{2\alpha}&=\frac{J_2^2}{qJ_7^3}\sum_{n=0}^{\infty}A(7^{2\alpha}n+\lambda_{2\alpha})q^n,\\
L_{2\alpha+1}&=\frac{qJ_{14}^2}{J_1^3}\sum_{n=0}^{\infty}A(7^{2\alpha+1}n+\lambda_{2\alpha+1})q^n,
\end{align*}
where 
$$
A(n):=a_f(n/7)-a_f(7n-2)
$$ and 
$$
\lambda_{2\alpha}=\lambda_{2\alpha+1}=\frac{7}{24}(1-7^{2\alpha}).
$$ 
Following \cite{Pa-Ra12} we call a 
map $a \,:\, \mathbb{Z} \longrightarrow \mathbb{Z}$ a 
\textit{discrete function} if it has finite support.
We define
\begin{align*}
&X_A\\
&:=
\left\{\sum_{k=0}^{\infty}r_1(k)7^{[\frac{7k}{4}]}t^k
   +p_0\sum_{k=0}^{\infty}r_2(k)7^{[\frac{7k}{4}]}t^k
   +p_1\sum_{k=1}^{\infty}r_3(k)7^{[\frac{7k-3}{4}]}t^k\,:\,
   \mbox{each $r_j$ is a discrete function}\right\},\\
&X_B\\
&:=\left\{\sum_{k=1}^{\infty}r_1(k)7^{[\frac{7k-5}{4}]}t^k
     +p_0\sum_{k=1}^{\infty}r_2(k)7^{[\frac{7k-5}{4}]}t^k
     +p_1\sum_{k=2}^{\infty}r_3(k)7^{[\frac{7k-8}{4}]}t^k\,:\,
   \mbox{each $r_j$ is a discrete function}\right\},
\end{align*}
We will prove that for $\alpha>0$:
\beq
\mylabel{eq:7-L2a}
L_{2\alpha}\in 7^\alpha X_A,
\eeq
where for a set $X$ and a number $k$
$$
kX:=\{kx:x\in X\}.
$$

Firstly from Appendix \ref{funr-7} we see that 
that in each case there is an integer $l$ and 
discrete functions $a_{k,u}^{(i)}(n)$ and $b_{k,u}^{(i)}(n)$ 
for $l \le k \le l + 6$ such that
\begin{align}
\mylabel{eq:UA7k}
U_A(ut^k)=&\sum_{n\geq [(k+7)/7]}a_{k,u}^{(0)}(n)7^{[\frac{7n-k-5}{4}]}t^n+p_0\sum_{n\geq [(k+7)/7]}a_{k,u}^{(1)}(n)7^{[\frac{7n-k-5}{4}]}t^n\\
\nonumber
&+\sum_{n\geq [(k+14)/7]}a_{k,u}^{(2)}(n)7^{[\frac{7n-k-8}{4}]}t^n,\\
\mylabel{eq:UB7k}
U_B(ut^k)=&\sum_{n\geq [k/7]}b_{k,u}^{(0)}(n)7^{[\frac{7n-k+5}{4}]}t^n+p_0\sum_{n\geq [k/7]}b_{k,u}^{(1)}(n)7^{[\frac{7n-k+5}{4}]}t^n\\
\nonumber
&+\sum_{n\geq [(k+6)/7]}b_{k,u}^{(2)}(n)7^{[\frac{7n-k+2}{4}]}t^n,
\end{align}
where $u$ is one of $1$, $p_0$ or $p_1$. 
Then using Lemma \ref{lem1} and Lemma \ref{lem2}, 
we find that \omyeqn{UA7k} and \omyeqn{UB7k} hold for all $k\in \mathbb{N}$. 
Next, we prove \omyeqn{7-L2a} inductively by proving the following
three statements:   
\begin{align*}
&L_1 \in X_B,\\
&g \in X_B \text{ implies } U_B(g)\in 7X_A, \quad \mbox{and}\\
&g \in X_A \text{ implies } U_A(g)\in X_B.
\end{align*}
Let $k=0$ in \omyeqn{UA7k} we can see that
\begin{align*}
L_1=&U_A(L_0)=7U_A(p_0)+U_A(p_1)\\
=&\sum_{n=1}^{\infty}r_1(n)7^{[\frac{7n-5}{4}]}t^n+p_0\sum_{n=1}^{\infty}r_2(n)7^{[\frac{7n-5}{4}]}t^n+\sum_{n=2}^{\infty}r_3(n)7^{[\frac{7n-8}{4}]}t^n \in X_B,
\end{align*}
with some discrete functions $r_i$. Assume that $g\in X_B$. There are discrete functions $r_i$ such that
\begin{align*}
g=\sum_{k=1}^{\infty}r_1(k)7^{[\frac{7k-5}{4}]}t^k+p_0\sum_{k=1}^{\infty}r_2(k)7^{[\frac{7k-5}{4}]}t^k+p_1\sum_{k=2}^{\infty}r_3(k)7^{[\frac{7k-8}{4}]}t^k.
\end{align*}
This implies that
\begin{align}
\mylabel{eq:7UBg}
U_B(g)=\sum_{k=1}^{\infty}r_1(k)7^{[\frac{7k-5}{4}]}U_B(t^k)+\sum_{k=1}^{\infty}r_2(k)7^{[\frac{7k-5}{4}]}U_B(p_0t^k)+\sum_{k=2}^{\infty}r_3(k)7^{[\frac{7k-8}{4}]}U_B(p_1t^k).
\end{align}
Each sum in \omyeqn{7UBg} can be written in the form $7g_1$ for 
some $g_1\in X_A$. Since the proofs are similar we only consider the first sum.
From \omyeqn{UB7k}
\begin{align}
\mylabel{eq:7ubup}
\sum_{k=1}^{\infty}r_1(k)7^{[\frac{7k-5}{4}]}U_B(t^k)=&\sum_{k=1}^{\infty}\sum_{n=0}^{\infty}r_1(k)(b_{k,1}^{(0)}(n)+b_{k,1}^{(1)}(n))7^{[\frac{7k-5}{4}]+[\frac{7n-k+5}{4}]}t^n\\
\nonumber
&+\sum_{k=1}^{\infty}\sum_{n=1}^{\infty}r_1(k)b_{k,1}^{(2)}(n)7^{[\frac{7k-5}{4}]+[\frac{7n-k+2}{4}]}t^n.
\end{align}
Observe that for $k=1$:
\begin{align*}
&\left[\frac{7k-5}{4}\right]+\left[\frac{7n-k+5}{4}\right]=\left[\frac{7n+4}{4}\right]=\left[\frac{7n}{4}\right]+1,\\
&\left[\frac{7k-5}{4}\right]+\left[\frac{7n-k+2}{4}\right]=\left[\frac{7n+1}{4}\right]=\left[\frac{7n-3}{4}\right]+1,
\end{align*}
and for $k>1$:
\begin{align*}
&\left[\frac{7k-5}{4}\right]+\left[\frac{7n-k+5}{4}\right]\geq\left[\frac{7n+6k-3}{4}\right]\geq\left[\frac{7n+9}{4}\right]\geq\left[\frac{7n}{4}\right]+1,\\
&\left[\frac{7k-5}{4}\right]+\left[\frac{7n-k+2}{4}\right]\geq\left[\frac{7n+6k-6}{4}\right]\geq\left[\frac{7n+6}{4}\right]\geq\left[\frac{7n-3}{4}\right]+1.
\end{align*}
Hence the right hand side of \omyeqn{7UBg} can be written in the form 
$7g_1$ for some $g_1\in X_A$. 
The statement that $g\in X_A$ implies $U_A(g)\in X_B$ can be proved analogously. 
So that we have proved \omyeqn{7-L2a} which implies that
\begin{align}
\mylabel{eq:7-A2a}
&A(7^{2\alpha}n+\lambda_{2\alpha})\equiv 0\pmod{7^\alpha},
\end{align}
and noting that $7^{2\alpha+1}n+\lambda_{2\alpha+1}$ is a subsequence of $7^{2\alpha}n+\lambda_{2\alpha}$, we have
\begin{align}
\mylabel{eq:7-A2a+1}
&A(7^{2\alpha+1}n+\lambda_{2\alpha+1})\equiv 0\pmod{7^\alpha}.
\end{align}
Congruences \omyeqn{7-A2a} and \omyeqn{7-A2a+1} are 
each cases of \omyeqn{rmod7} after replacing $2\alpha$ and $2\alpha+1$ by $\alpha$, 
where noting that for $\alpha>0$:
\begin{align*}
A(7^{2\alpha}n+\lambda_{2\alpha})=&a_f(7^{2\alpha-1}n+\lambda_{2\alpha}/7)-a_f(7^{2\alpha+1}n+7\lambda_{2\alpha}-2)\\
=&a_f(7^{2\alpha-1}n+\delta_{2\alpha-1})-a_f(7^{2\alpha+1}n+\delta_{2\alpha+1}),
\end{align*}
and
\begin{align*}
A(7^{2\alpha+1}n+\lambda_{2\alpha+1})=&a_f(7^{2\alpha}n+\lambda_{2\alpha+1}/7)-a_f(7^{2\alpha+2}n+7\lambda_{2\alpha+1}-2)\\
=&a_f(7^{2\alpha}n+\delta_{2\alpha})-a_f(7^{2\alpha+2}n+\delta_{2\alpha+2}).
\end{align*}
This completes the proof of  Theorem \thm{mainthm}.

\section{The crank parity function modulo powers of $7$}
\omylabel{sec:crankparity7}
\subsection{Preliminary Lemmas}
We denote
\begin{align*}
A:=\frac{\eta(\tau)^3\eta(98\tau)^2}{\eta(2\tau)^2\eta(49\tau)^3}=\frac{q^2J_1^3J_{98}^2}{J_2^2J_{49}^3},
\end{align*}
and for any $g:\mathbb{H}\rightarrow \mathbb{C}$, we define
\begin{align*}
U^{(1)}(g):=U_7(Ag),\text{ }U^{(0)}(g):=U_7(g),
\end{align*}
where $U_p(f)$ is defined by \eqn{Updef}.
Let $L_0:=1$ and for $\alpha\geq0$,
\begin{align*}
L_{2\alpha+1}:=U^{(1)}(L_{2\alpha}), \text{ }L_{2\alpha+2}:=U^{(0)}(L_{2\alpha+1}).
\end{align*}

Using \omyeqn{Updeffls} and \omyeqn{u71}, it is easy to verify that 
\begin{align*}
L_{2\alpha}
=&\frac{J_2^2}{J_1^3}\sum_{n=0}^{\infty}\beta(7^{2\alpha} n+\delta_{2\alpha})q^n,
\quad\mbox{for $\alpha\ge1$ and}\\
L_{2\alpha+1}
=&\frac{J_{14}^2}{J_7^3}\sum_{n=0}^{\infty}\beta(7^{2\alpha+1} n+\delta_{2\alpha+1})q^n,
\end{align*}
for $\alpha\ge0$.
Since $\frac{J_2^2}{J_1^3}$ and $\frac{J_{14}^2}{J_7^3}$ have leading coefficient $1$, 
the congruence \eqn{betamod7} is equivalent to
\begin{align}
L_{\alpha+2}\equiv K_\alpha L_\alpha\pmod{7^\alpha}.
\mylabel{eq:main1}
\end{align}

In order to prove \omyeqn{main1}, 
we use the forty-two fundamental relations in Appendix \ref{funcr-7}.
Again these identities can be
proved using the algorithm described in 
\cite[Section 2C, pp.8-9]{Ch-Ch-Ga20}.
We note  $t=\eta(7\tau)^4/\eta(\tau)^4$ and
\begin{align*}
p_0:=\frac{\eta(14\tau)^4\eta(\tau)^4}{\eta(7\tau)^4\eta(2\tau)^4}=\frac{qJ_{14}^4J_1^4}{J_7^4J_2^4},
\end{align*}
\begin{align*}
p_1:=\frac{1}{7}\left(\frac{\eta(14\tau)\eta(\tau)^7}{\eta(7\tau)\eta(2\tau)^7}-1\right)=\frac{1}{7}\left(\frac{J_{14}J_1^7}{J_7J_2^7}-1\right).
\end{align*}
It is clear the $q$-expansion of $p_0$ has integer coefficients.          
Let $u(q):=\frac{J_1}{J_2}$. Since $u(q^7)\equiv u(q)^7\pmod 7$, 
the $q$-expansion of $p_1=\frac{1}{7}(\frac{u(q)^7}{u(q^7)}-1)$ 
also has integer coefficients. 
To prove Theorem \thm{crankthm7}, some lemmas are needed.

\begin{lemma}
For $j=0$, $1$, $2$ there exist discrete functions of $n$, 
$a_{k,i}(n,j)$ and $b_{k,i}(n,j)$ such that
\begin{align}
U^{(1)}(p_0t^k)=&p_0\sum_{n\geq[k/7]}a_{k,0}(n,0)7^{[\frac{7n-k+1}{4}]}t^n+p_1\sum_{n\geq[k/7]}a_{k,0}(n,1)7^{[\frac{7n-k+5}{4}]}t^n
\mylabel{eq:uak}\\
&+\sum_{n\geq[(k+7)/7]}a_{k,0}(n,2)7^{[\frac{7n-k+1}{4}]}t^n,
\nonumber
\end{align}
\begin{align}
U^{(1)}(p_1t^k)=&p_0\sum_{n\geq[k/7]}a_{k,1}(n,0)7^{[\frac{7n-k+1}{4}]}t^n+p_1\sum_{n\geq[k/7]}a_{k,1}(n,1)7^{[\frac{7n-k+4}{4}]}t^n
\mylabel{eq:uak1}\\
&+\sum_{n\geq[(k+7)/7]}a_{k,1}(n,2)7^{[\frac{7n-k+1}{4}]}t^n,
\nonumber
\end{align}
\begin{align}
\mylabel{eq:uak2}
U^{(1)}(t^k)=&p_0\sum_{n\geq[k/7]}a_{k,2}(n,0)7^{[\frac{7n-k+1}{4}]}t^n+p_1\sum_{n\geq[k/7]}a_{k,2}(n,1)7^{[\frac{7n-k+5}{4}]}t^n\\
&+\sum_{n\geq[(k+7)/7]}a_{k,2}(n,2)7^{[\frac{7n-k+1}{4}]}t^n,
\nonumber
\end{align}
\begin{align}
\mylabel{eq:ubk}
U^{(0)}(p_0t^k)=&p_0\sum_{n\geq[(k+7)/7]}b_{k,0}(n,0)7^{[\frac{7n-k-1}{4}]}t^n+p_1\sum_{n\geq[k/7]}b_{k,0}(n,1)7^{[\frac{7n-k+2}{4}]}t^n\\
&+\sum_{n\geq[(k+7)/7]}b_{k,0}(n,2)7^{[\frac{7n-k-1}{4}]}t^n,
\nonumber
\end{align}
\begin{align}
\mylabel{eq:ubk1}
U^{(0)}(p_1t^k)=&p_0\sum_{n\geq[(k+7)/7]}b_{k,1}(n,0)7^{[\frac{7n-k-1}{4}]}t^n+p_1\sum_{n\geq[k/7]}b_{k,1}(n,1)7^{[\frac{7n-k+2}{4}]}t^n\\
&+\sum_{n\geq[(k+7)/7]}b_{k,1}(n,2)7^{[\frac{7n-k-1}{4}]}t^n,
\nonumber
\end{align}
\begin{align}
U^{(0)}(t^k)=\sum_{n\geq[(k+6)/7]}b_{k,2}(n,2)7^{[\frac{7n-k-1}{4}]}t^n.
\mylabel{eq:uck}
\end{align}
\end{lemma}

\begin{proof}
Firstly from Appendix \ref{funr-7} we see that 
that in each case there is an integer $l$ and 
discrete functions $a_{k,u}^{(i)}(n)$ and $b_{k,u}^{(i)}(n)$ 
for $l \le k \le l + 6$ such that
From Appendix \ref{funcr-7} we see that  
in each of \omyeqn{uak}-\omyeqn{uck} 
there is an integer $l$ and appropriate discrete functions for $l \le k \le l + 6$. 
Since each sum in 
\omyeqn{uak}-\omyeqn{uck}  is finite, by Lemma \ref{lem1} and 
Lemma \ref{lem2}, we have that \omyeqn{uak}-\omyeqn{uck} hold for all $k\in \mathbb{N}$.
\end{proof}

\begin{lemma}
\mylabel{lem:La}
For each $\alpha$ and $i=0$, $1$, $2$ there exist unique polynomials 
$P_i^{(\alpha)}(t)$ with 
integer coefficients, such that
\begin{align}
L_{\alpha}=p_0P_0^{(\alpha)}(t)+p_1P_1^{(\alpha)}(t)+P_2^{(\alpha)}(t).
\mylabel{eq:La}
\end{align}
\end{lemma}
\begin{proof}
From $L_0=1$ and \omyeqn{uak}-\omyeqn{uck}, the existence of \omyeqn{La} 
is obvious. We can check that $t$, $p_0$ and $7p_1+1$ are modular functions 
on $\Gamma_0(14)$ by Theorem \thm{etamodthm}. Using Theorem \thm{ordthm}, 
Theorem \thm{chualang} and Lemma \lem{fanw} we can calculate the order of 
$t$, $p_0$ and $p_1$ at the cusps  $0$ and $1/2$.
\begin{align*}
&\ORD(t,0,\Gamma_0(14))=-2,\quad \ORD(t,1/2,\Gamma_0(14))=-1,\\
&\ORD(p_0,0,\Gamma_0(14))=1,\quad \ORD(p_0,1/2,\Gamma_0(14))=-1,\\
&\ORD(p_1,0,\Gamma_0(14))=0,\quad \ORD(p_1,1/2,\Gamma_0(14))=-2.
\end{align*}
Suppose that
\beq
\mylabel{eq:liner-aN}
p_0\sum_{n=0}^{M}a_nt^n+p_1\sum_{n=0}^{M}b_nt^n+\sum_{n=0}^{N}c_nt^n=0,
\eeq
and $c_N\neq 0$. 
We see 
that the order of
$$
\sum_{n=0}^{N}c_nt^n,
$$
at $0$ is $-2N$, and so is the order of
$$
p_0\sum_{n=0}^{M}a_nt^n+p_1\sum_{n=0}^{M}b_nt^n,
$$
at $0$. 
Each term $p_0t^n$ has order $1-2n$ and each term $p_1t^m$ has the different 
order  $-2m$
at $0$. This implies $b_{N}\neq 0$ and for $n,m>N$ we have $a_n=b_m=0$. 
The order at $1/2$ of
$$
p_0\sum_{n=0}^{N}a_nt^n+p_1\sum_{n=0}^{N}b_nt^n,
$$
is $-2-N$ but for 
$$
\sum_{n=0}^{N}c_nt^n,
$$
it is $-N$. This is a contradiction. It implies that there is no $n$ 
such that 
$c_n\neq 0$ and \omyeqn{liner-aN} holds. This means that   
$$
p_0\sum_{n=0}^{M}a_nt^n+p_1\sum_{n=0}^{M}b_nt^n=0.
$$
But the order of each term at $0$ is different. 
We have $a_n=b_n=0$ for all $n$, which means for each $N>0$ 
the functions $1, t,..., t^N, p_0, p_0t,..., p_0t^N, p_1, p_1t,..., p_1t^N$ are 
linear independent. Hence the expression in \omyeqn{La} is unique.
\end{proof}

We need lower-bounds for the $7$-adic order of
coefficients. 
\begin{lemma}
\mylabel{lem:L2a}
For $\alpha \geq 1$ there exist integers $d_{n,i}^{(\alpha)}$, $i=0,1,2$ such that
\begin{align}
L_{2\alpha-1}=p_0\sum_{k=0}^{\infty}d_{k,0}^{(2\alpha-1)}7^{[\frac{7k+1}{4}]}t^k+p_1\sum_{k=0}^{\infty}d_{k,1}^{(2\alpha-1)}7^{[\frac{7k+4}{4}]}t^k+\sum_{k=1}^{\infty}d_{k,2}^{(2\alpha-1)}7^{[\frac{7k+1}{4}]}t^k,
\mylabel{eq:L2a-1}
\end{align}
and
\begin{align}
L_{2\alpha}=p_0\sum_{k=1}^{\infty}d_{k,0}^{(2\alpha)}7^{[\frac{7k-1}{4}]}t^k+p_1\sum_{k=0}^{\infty}d_{k,1}^{(2\alpha)}7^{[\frac{7k+2}{4}]}t^k+\sum_{k=1}^{\infty}d_{k,2}^{(2\alpha)}7^{[\frac{7k-1}{4}]}t^k.
\mylabel{eq:L2a}
\end{align}
\end{lemma}

\begin{proof}
From Appendix \ref{funcr-7} we see that 
$L_1=2p_0+7p_1$ has form given in \omyeqn{L2a-1}. 
Assume that $L_{2\alpha-1}$ has the for given in \omyeqn{L2a-1} for a fixed $\alpha$, then
\begin{align}
\mylabel{eq:sum1}
L_{2\alpha}=&U^{(0)}(L_{2\alpha-1})\\
\nonumber
&
=\sum_{k=0}^{\infty}d_{k,0}^{(2\alpha-1)}7^{[\frac{7k+1}{4}]}U^{(0)}(p_0t^k)
 +\sum_{k=0}^{\infty}d_{k,1}^{(2\alpha-1)}7^{[\frac{7k+4}{4}]}U^{(0)}(p_1t^k)\\
&+\sum_{k=1}^{\infty}d_{k,2}^{(2\alpha-1)}7^{[\frac{7k+1}{4}]}U^{(0)}(t^k).
\nonumber
\end{align}
We  wish to show that the form of each sum on the right side of \omyeqn{sum1} 
satisfies the form in \omyeqn{L2a}. Since the proofs are similar we
just consider the first sum.
From \omyeqn{ubk},
\begin{align*}
&\sum_{k=0}^{\infty}d_{k,0}^{(2\alpha-1)}7^{[\frac{7k+1}{4}]}U^{(0)}(p_0t^k)\\
=&p_0\sum_{k=0}^{\infty}\sum_{n\geq[(k+7)/7]}b_0(n,k;p_0)d_{k,0}^{(2\alpha-1)}7^{[\frac{7n-k-1}{4}]+[\frac{7k+1}{4}]}t^n\\
&+p_1\sum_{k=0}^{\infty}\sum_{n\geq[k/7]}b_1(n,k;p_0)d_{k,0}^{(2\alpha-1)}7^{[\frac{7n-k+2}{4}]+[\frac{7k+1}{4}]}t^n\\
&+\sum_{k=0}^{\infty}\sum_{n\geq[(k+7)/7]}b_2(n,k;p_0)d_{k,0}^{(2\alpha-1)}7^{[\frac{7n-k-1}{4}]+[\frac{7k+1}{4}]}t^n.
\end{align*}
For $k=0$, we have
\begin{align*}
\left[\frac{7n-k-1}{4}\right]+\left[\frac{7k+1}{4}\right]=\left[\frac{7n-1}{4}\right],
\end{align*}
\begin{align*}
\left[\frac{7n-k+2}{4}\right]+\left[\frac{7k+1}{4}\right]=\left[\frac{7n+2}{4}\right],
\end{align*}
and for $k\geq 1$, we have
\begin{align*}
\left[\frac{7n-k-1}{4}\right]+\left[\frac{7k+1}{4}\right]\geq\left[\frac{7n-k-1+7k+1-3}{4}\right]\geq\left[\frac{7n-1}{4}\right],
\end{align*}
\begin{align*}
\left[\frac{7n-k+2}{4}\right]+\left[\frac{7k+1}{4}\right]\geq\left[\frac{7n-k+2+7k+1-3}{4}\right]\geq\left[\frac{7n+2}{4}\right].
\end{align*}
So that the first sum in the right side of \omyeqn{sum1} has the form of \omyeqn{L2a}. 
Similarly, also the second and third sums have the correct form. 
Hence $L_{2\alpha}$ has the desired form.    
The proof that the correct form of $L_{2\alpha}$ implies 
the correct form of $L_{2\alpha+1}$ is analogous.
The general result follows by induction.
\end{proof}

By Lemmas \lem{La} and \lem{L2a} we let
\begin{align*}
P_i^{(\alpha)}(t):=\sum_{n=0}^{\infty}l_{n,i}^{(\alpha)}t^n,
\end{align*}
for $i=0,1,2$ be the unique polynomials such that
\begin{align*}
L_{\alpha}=p_0\sum_{n=0}^{\infty}l_{n,0}^{(\alpha)}t^n+p_1\sum_{n=0}^{\infty}l_{n,1}^{(\alpha)}t^n+\sum_{n=0}^{\infty}l_{n,2}^{(\alpha)}t^n.
\end{align*}
Next we define
\begin{align*}
D^{(\alpha)}(l_{m,i},l_{n,j}):=l_{m,i}^{(\alpha)}l_{n,j}^{(\alpha+2)}-l_{m,i}^{(\alpha+2)}l_{n,j}^{(\alpha)},
\end{align*}
and denote $\pi(n)$ be the $7$-adic order of $n$ (i.e.\ the highest power of $7$
that divides $n$).

\begin{lemma}
\mylabel{lem:lemD}
For $\alpha \geq 1$, $i,j=0,1,2$,
\begin{align}
\pi(D^{(2\alpha-1)}(l_{m,i},l_{n,j}))&\geq 2\alpha-2+m+n+\max(\lambda_i,\lambda_j),
\mylabel{eq:D2a-1}\\
\pi(D^{(2\alpha)}(l_{m,i},l_{n,j}))&\geq 2\alpha-1+m+n+\lambda_i\lambda_j,
\mylabel{eq:D2a}
\end{align}
where $\lambda_0=\lambda_2=0$ and $\lambda_1=1$.
\end{lemma}

\begin{proof}
Since $l_{n,i}^{(1)}=0$ except $l_{0,0}^{(1)}=2$ and $l_{0,1}^{(1)}=7$, from \omyeqn{L2a-1} we have
\begin{align*}
\pi(D^{(1)}(l_{m,i},l_{n,j}))=\infty,
\end{align*}
when $m,n>0$ and
\begin{align*}
\pi(D^{(1)}(l_{0,i},l_{n,j}))=&\pi(-D^{(1)}(l_{n,j},l_{0,i}))=\pi(l_{0,i}^{(1)}l_{n,j}^{(3)})\\
\geq& \lambda_i+\left[\frac{7n+\mu_j}{4}\right]=\lambda_i+n+\left[\frac{3n+\mu_j}{4}\right]\geq n+\max(\lambda_i,\lambda_j),
\end{align*}
when $n>0$ where $\mu_0=\mu_2=1$ and $\mu_1=4$ and
\begin{align*}
\pi(D^{(1)}(l_{0,0},l_{0,1}))=\pi(l_{0,0}^{(1)}l_{0,1}^{(3)}-l_{0,0}^{(3)}l_{0,1}^{(1)})\geq \min(l_{0,1}^{(1)},l_{0,1}^{(3)})\geq 1=\max(\lambda_0,\lambda_1),
\end{align*}
which proves \omyeqn{D2a-1} for $\alpha=1$. Suppose now that \omyeqn{D2a-1} holds for 
fixed $\alpha$. We will compare the coefficients on both sides of
\begin{align*}
&p_0\sum_{n=0}^{\infty}l_{n,0}^{(2\alpha)}t^n+p_1\sum_{n=0}^{\infty}l_{n,1}^{(2\alpha)}t^n+\sum_{n=0}^{\infty}l_{n,2}^{(2\alpha)}t^n\\
=&\sum_{k=0}^{\infty}l_{k,0}^{(2\alpha-1)}U^{(0)}(p_0t^k)+\sum_{k=0}^{\infty}l_{k,1}^{(2\alpha-1)}U^{(0)}(p_1t^k)+\sum_{k=0}^{\infty}l_{k,2}^{(2\alpha-1)}U^{(0)}(t^k).
\end{align*}
For convenience, denote $p_2:=1$ and for $u=0,1,2$ let
\begin{align*}
U^{(0)}(p_ut^k)=p_0\sum_{n}x_{k,u}(n,0)t^n+p_1\sum_{n}x_{k,u}(n,1)t^n+\sum_{n}x_{k,u}(n,2)t^n.
\end{align*}
We have
\begin{align*}
l_{n,i}^{(2\alpha)}=\sum_{k,u}x_{k,u}(n,i)l_{k,u}^{(2\alpha-1)},
\end{align*}
which shows that
\begin{align}
D^{(2\alpha)}(l_{m,i},l_{n,j})=\sum_{k,r,u,v}x_{k,u}(m,i)x_{r,v}(n,j)D^{(2\alpha-1)}(l_{k,u},l_{r,v}).
\mylabel{eq:lni}
\end{align}
From \omyeqn{ubk}-\omyeqn{uck}, we have $\pi(x_{k,u}(m,i))\geq [\frac{7m-k+\nu_i}{4}]$, where $\nu_0=\nu_2=-1$ and $\nu_1=2$. So that
\begin{align}
\mylabel{eq:pid}
&\pi(D^{(2\alpha)}(l_{m,i},l_{n,j}))\\
\nonumber
=&\pi\left(\sum_{k,r,u,v}x_{k,u}(m,i)x_{r,v}(n,j)D^{(2\alpha-1)}(l_{k,u},l_{r,v})\right)\\
\nonumber
\geq&\min_{(k,u)\neq (r,v)}\left(k+r+2\alpha-2+\max(\lambda_u,\lambda_v)+\left[\frac{7m-k+\nu_i}{4}\right]+\left[\frac{7n-r+\nu_j}{4}\right]\right)\\
=&\min_{(k,u)\neq (r,v)}\left(m+n+2\alpha-2+\max(\lambda_u,\lambda_v)+\left[\frac{3m+3k+\nu_i}{4}\right]+\left[\frac{3n+3r+\nu_j}{4}\right]\right).
\nonumber
\end{align}
Noting \omyeqn{L2a}, $l_{0,i}^{(2\alpha)}\neq 0$ only if $i=1$, which is $3m+\nu_i>1$ and also $3n+\nu_j>1$. Further more, if $i=j=1$ then $m>0$ or $n>0$ which is $3m+\nu_i\geq 5$ or $3n+\nu_j\geq 5$. Since $i=j=1$ if and only if $\lambda_i\lambda_j=1$,
\begin{align}
\left[\frac{3m+3k+\nu_i}{4}\right]+\left[\frac{3n+3r+\nu_j}{4}\right]\geq\left[\frac{3k+1}{4}\right]+\left[\frac{3r+1}{4}\right]+\lambda_i\lambda_j.
\mylabel{eq:pid1}
\end{align}
Also noting \omyeqn{L2a-1}, $l_{0,2}^{(2\alpha-1)}=0$. This implies that if $k=r=0$, then one of $u$ or $v$ is $1$, which means one of $\lambda_u$ or $\lambda_v$ is $1$. So that
\begin{align}
\max(\lambda_u,\lambda_v)+\left[\frac{3k+1}{4}\right]+\left[\frac{3r+1}{4}\right]\geq 1.
\mylabel{eq:pid2}
\end{align}
From \omyeqn{pid}-\omyeqn{pid2}
\begin{align*}
\pi(D^{(2\alpha)}(l_{m,i},l_{n,j}))&\geq 2\alpha-1+m+n+\lambda_i\lambda_j,
\end{align*}
which is \omyeqn{D2a} hold for $\alpha$. Then suppose that \omyeqn{D2a} holds for $\alpha$. Let
\begin{align*}
U^{(1)}(p_ut^k)=p_0\sum_{n}y_{k,u}(n,0)t^n+p_1\sum_{n}y_{k,u}(n,1)t^n+\sum_{n}y_{k,u}(n,2)t^n.
\end{align*}
Similar to \omyeqn{lni}, we have
\begin{align*}
D^{(2\alpha+1)}(l_{m,i},l_{n,j})=\sum_{k,r,u,v}y_{k,u}(m,i)y_{r,v}(n,j)D^{(2\alpha)}(l_{k,u},l_{r,v}).
\end{align*}
From \omyeqn{uak}-\omyeqn{uak2}, we have $\pi(y_{k,u}(m,i))\geq [\frac{7m-k+\mu_{i,u}}{4}]$, where $\mu_{0,u}=\mu_{2,u}=1$ and $\mu_{1,u}=5-\lambda_u$. So that
\begin{align}
\mylabel{eq:apid}
&\pi(D^{(2\alpha+1)}(l_{m,i},l_{n,j}))\\
\nonumber
=&\pi\left(\sum_{k,r,u,v}y_{k,u}(m,i)y_{r,v}(n,j)D^{(2\alpha)}(l_{k,u},l_{r,v})\right)\\
\nonumber
\geq&\min_{(k,u)\neq (r,v)}\left(k+r+2\alpha-1+\lambda_u\lambda_v+\left[\frac{7m-k+\mu_{i,u}}{4}\right]+\left[\frac{7n-r+\mu_{j,v}}{4}\right]\right)\\
=&\min_{(k,u)\neq (r,v)}\left(m+n+2\alpha-1+\lambda_u\lambda_v+\left[\frac{3m+3k+\mu_{i,u}}{4}\right]+\left[\frac{3n+3r+\mu_{j,v}}{4}\right]\right).
\nonumber
\end{align}
Also from \omyeqn{L2a} and then $l_{0,u}^{(2\alpha)}\neq 0$ only if $u=1$, we can see that $k\geq 1$ or $r\geq 1$ which is $3k+\mu_{i,u}\geq 4$ or $3r+\mu_{j,v}\geq 4$. When $u=v=1$,
\begin{align}
\lambda_u\lambda_v+\left[\frac{3m+3k+\mu_{i,u}}{4}\right]+\left[\frac{3n+3r+\mu_{j,v}}{4}\right]\geq 2\geq \max(\lambda_i,\lambda_j)+1.
\mylabel{eq:ij1}
\end{align}
When $u\neq 1$,($v\neq 1$ are analogous), then $k\geq 1$ and
\begin{align}
\lambda_u\lambda_v+\left[\frac{3m+3k+\mu_{i,u}}{4}\right]+\left[\frac{3n+3r+\mu_{j,v}}{4}\right]\geq\left[\frac{3m+3+\mu_{i,u}}{4}\right]+\left[\frac{3n+\mu_{j,v}}{4}\right],
\mylabel{eq:ij2}
\end{align}
if $j=1$, then $\mu_{j,v}\geq 4$ and
\begin{align}
\left[\frac{3m+3+\mu_{i,u}}{4}\right]+\left[\frac{3n+\mu_{j,v}}{4}\right]\geq 2\geq \max(\lambda_i,\lambda_j)+1,
\mylabel{eq:uv1}
\end{align}
if $i=1$, then $\mu_{i,u}\geq 5$ and
\begin{align}
\left[\frac{3m+3+\mu_{i,u}}{4}\right]+\left[\frac{3n+\mu_{j,v}}{4}\right]\geq 2\geq \max(\lambda_i,\lambda_j)+1,
\mylabel{eq:uv2}
\end{align}
if $i,j\neq 1$,
\begin{align}
\left[\frac{3m+3+\mu_{i,u}}{4}\right]+\left[\frac{3n+\mu_{j,v}}{4}\right]\geq 1\geq \max(\lambda_i,\lambda_j)+1.
\mylabel{eq:uv3}
\end{align}
From \omyeqn{uv1}-\omyeqn{uv3} we have \omyeqn{ij2}, and from \omyeqn{apid}-\omyeqn{ij2} we have
\begin{align*}
\pi(D^{(2\alpha+1)}(l_{m,i},l_{n,j}))&\geq 2\alpha+m+n+\max(\lambda_i,\lambda_j),
\end{align*}
which is \omyeqn{D2a-1} hold for $\alpha+1$. Hence \omyeqn{D2a-1} and \omyeqn{D2a} hold for all $\alpha\geq 1$.
\end{proof}

\subsection{Prove of  Theorem \thm{crankthm7}}
\begin{proof}
From Lemma \lem{lemD}, for each $(n,j)\neq(0,0)$, by $l_{0,2}^{(2\alpha-1)}=0$ we have
\begin{align}
\pi(D^{(2\alpha-1)}(l_{0,0},l_{n,j}))&\geq 2\alpha-1,
\mylabel{eq:mainr}
\end{align}
and for $(n,j)\neq(0,1)$, by $l_{0,0}^{(2\alpha)}=0$ and $l_{0,2}^{(2\alpha)}=0$ we have
\begin{align}
\pi(D^{(2\alpha)}(l_{0,1},l_{n,j}))&\geq 2\alpha.
\label{mainr1}
\end{align}
It is easy to see that for $\alpha\geq 1$,
\begin{align*}
L_{2\alpha-1}\equiv l_{0,0}^{(2\alpha-1)}p_0\pmod 7,
\end{align*}
and
\begin{align*}
L_{2\alpha}\equiv l_{0,1}^{(2\alpha)}p_1\pmod 7.
\end{align*}
Then we notice that $7\nmid l_{0,0}^{(2\alpha-1)}$ and $7\nmid l_{0,1}^{(2\alpha)}$ which can be implied from $l_{0,0}^{(1)}=2$, $U^{(0)}(p_0)\equiv 3p_1\pmod 7$ and $U^{(1)}(p_1)\equiv 2p_0\pmod 7$. Let $x_{2\alpha-1}$ be a solution of
\begin{align*}
l_{0,0}^{(2\alpha-1)}\equiv xl_{0,0}^{(2\alpha+1)}\pmod{7^{2\alpha-1}},
\end{align*}
then for $(n,i)\neq(0,0)$, from \omyeqn{mainr},
\begin{align*}
l_{n,i}^{(2\alpha-1)}l_{0,0}^{(2\alpha-1)}\equiv x_{2\alpha-1}l_{0,0}^{(2\alpha+1)}l_{n,i}^{(2\alpha-1)}\equiv x_{2\alpha-1}l_{0,0}^{(2\alpha-1)}l_{n,i}^{(2\alpha+1)}\pmod{7^{2\alpha-1}}.
\end{align*}
Cancelling $l_{0,0}^{(2\alpha-1)}$ we obtain
\begin{align}
l_{n,i}^{(2\alpha-1)}\equiv x_{2\alpha-1}l_{n,i}^{(2\alpha+1)}\pmod{7^{2\alpha-1}},
\mylabel{eq:lni2a-1}
\end{align}
Similarly, let $x_{2\alpha}$ be a solution of
\begin{align*}
l_{0,1}^{(2\alpha)}\equiv xl_{0,1}^{(2\alpha)}\pmod{7^{2\alpha}},
\end{align*}
then for $(n,i)\neq(0,1)$, from \eqref{mainr1},
\begin{align*}
l_{n,i}^{(2\alpha)}l_{0,1}^{(2\alpha)}\equiv x_{2\alpha}l_{0,1}^{(2\alpha+2)}l_{n,i}^{(2\alpha)}\equiv x_{2\alpha}l_{0,1}^{(2\alpha)}l_{n,i}^{(2\alpha+2)}\pmod{7^{2\alpha}}.
\end{align*}
Cancelling $l_{0,1}^{(2\alpha)}$ we obtain
\begin{align}
l_{n,i}^{(2\alpha)}\equiv x_{2\alpha}l_{n,i}^{(2\alpha+2)}\pmod{7^{2\alpha}}.
\mylabel{eq:lni2a}
\end{align}
From \omyeqn{La}, \omyeqn{lni2a-1} and \omyeqn{lni2a} prove \omyeqn{main1} and also \omyeqn{betamod7}.
\end{proof}

\section{The congruence  modulo powers of $5$ and $7$ for $\omega(q)$}
\omylabel{sec:w-5-7}
\subsection{The $U_P^{*}$ operator}
Let $p$ prime and $f$ be a function defined on the upper-half plane
(not necessarily a modular function).
We define  
\begin{align*}
\stroke{f}{U_p^*} 
:= \frac{1}{p} \sum_{j=0}^{p-1} \stroke{f}{\MAT{1/p}{24j/p}{0}{1}}
 = \frac{1}{p} \sum_{j=0}^{p-1} f\left(\frac{\tau +24j}{p}\right),
\end{align*}
and we note that for a function $g(\tau)$
$$
\stroke{g}{U_p^*} =\stroke{g^*}{U_p}(\tau/24) ,
$$
where $g^*(\tau)=g(24\tau)$.

\subsection{Modular properties}
Consider the following third order mock theta functions:
\begin{align*}
f(q):=&\sum_{n=0}^{\infty}\frac{q^{n^2}}{(-q;q)_n^2},\\
\omega(q):=&\sum_{n=0}^{\infty}\frac{q^{2n^2+2n}}{(q;q^2)_n^2}.
\end{align*}
Following \cite{Zw01}, we define $F:=(f_0, f_1, f_2)^T$ by:
\begin{align*}
f_0(\tau):=&q^{-1/24}f(q),\\
f_1(\tau):=&2q^{1/3}\omega(q^{\frac{1}{2}}),\\
f_2(\tau):=&2q^{1/3}\omega(q^{-\frac{1}{2}}).
\end{align*}

Also define
\begin{align*}
G(\tau):=2i\sqrt{3}\int_{-\overline{\tau}}^{i\infty}\frac{(g_1(z), g_0(z), -g_2(z))^T}{\sqrt{-i(z+\tau)}}dz,
\end{align*}
where
\begin{align*}
g_0(z):=&\sum_{n\in \mathbb{Z}}(-1)^n(n+1/3)e^{3\pi i(n+1/3)^2z},\\
g_1(z):=&-\sum_{n\in \mathbb{Z}}(n+1/6)e^{3\pi i(n+1/6)^2z},\\
g_2(z):=&\sum_{n\in \mathbb{Z}}(n+1/3)e^{3\pi i(n+1/3)^2z}.
\end{align*}

Letting $H=(h_0, h_1, h_2)^T:=F-G$, \cite[Theorem 3.6]{Zw01} gives
\begin{align}
\mylabel{eq:htau1}
H(\tau+1)=
\left(\begin{matrix}
\zeta_{24}^{-1} &0 &0\\
0 &0 &\zeta_3\\
0 &\zeta_3 &0
\end{matrix}\right)
H(\tau),
\end{align}
where $\zeta_m:=e^{2\pi i/m}$, and
\begin{align}
\mylabel{eq:htau2}
\frac{1}{\sqrt{-i\tau}}H(-1/\tau)=
\left(\begin{matrix}
0 &1 &0\\
1 &0 &0\\
0 &0 &-1
\end{matrix}\right)
H(\tau).
\end{align}

\begin{lemma}\label{lemUpG}
For each prime $p\geq 5$,
$$
\stroke{G}{U_p^*}=\chi_6(p)G(p\tau),
$$
where
$$\chi_6(p)=
\begin{cases}
1 & p\equiv 1 \pmod{6},\\
-1 & p\equiv -1 \pmod{6},\\
0 & \text{otherwise}.
\end{cases}
$$
\end{lemma}
\begin{proof}
We prove that for $i=0, 1, 2$,
\begin{align}
\mylabel{eq:gi0}
\stroke{g_i}{U_p^*}=p\chi_6(p)g_i(p\tau).
\end{align}
Since the proofs are similar, we only prove the case $i=0$. By definition,
$$
g_0(\tau):=\sum_{n\in \mathbb{Z}}(-1)^n(n+1/3)e^{3\pi i(n+1/3)^2\tau}.
$$
So that
\begin{align}
\mylabel{eq:gi1}
\stroke{g_0}{U_p^*}=\frac{1}{p} \sum_{j=0}^{p-1}\sum_{n\in \mathbb{Z}}(-1)^n(n+1/3)e^{3\pi i(n+1/3)^2\tau/p+3\pi i(n+1/3)^2\cdot 24j/p}.
\end{align}
Noting that $3\pi i(n+1/3)^2\cdot 24j=2 \pi i \cdot 4j(3n+1)^2$ and $4(3n+1)^2\equiv 0 \pmod p$ if and only if $3n+1\equiv 0 \pmod p$, which implies
\begin{align}
\mylabel{eq:gi2}
&\sum_{j=0}^{p-1}\sum_{\substack{n\in \mathbb{Z}\\p\nmid 3n+1}}(-1)^n(n+1/3)e^{3\pi i(n+1/3)^2\tau/p+3\pi i(n+1/3)^2\cdot 24j/p}\\
\nonumber
=&\sum_{\substack{n\in \mathbb{Z}\\p\nmid 3n+1}}\sum_{j=0}^{p-1}\zeta_p^{4j(3n+1)^2}(-1)^n(n+1/3)e^{3\pi i(n+1/3)^2\tau/p}=0.
\end{align}
Letting $3n+1=kp$ then $k=3m^*+\chi_6(p)$ for $m^*\in \mathbb{Z}$, and letting $m=\chi_6(p)m^*$ then $n+1/3=p\chi_6(p)(m+1/3)$ and $(-1)^n=(-1)^m$, by \omyeqn{gi1} and \omyeqn{gi2} we have
\begin{align*}
\stroke{g_0}{U_p^*}=&\sum_{3n+1\equiv 0 \pmod p}(-1)^n(n+1/3)e^{3\pi i(n+1/3)^2\tau/p}\\
=&\sum_{m\in \mathbb{Z}}(-1)^mp\chi_6(p)(m+1/3)e^{3\pi i(m+1/3)^2p\tau}\\
=&p\chi_6(p)g_0(p\tau).
\end{align*}
Noting that for $\tau=x+yi$
\begin{align}
\mylabel{eq:Gtau}
G(\tau)=&2i\sqrt{3}\int_{-\overline{\tau}}^{i\infty}\frac{(g_1(z), g_0(z), -g_2(z))^T}{\sqrt{-i(z+\tau)}}dz\\
\nonumber
=&-2\sqrt{3}\int_{y}^{\infty}\frac{(g_1(-x+it), g_0(-x+it), -g_2(-x+it))^T}{\sqrt{y+t}}dt.
\end{align}
Using \omyeqn{gi0}, \omyeqn{Gtau} and $g_i(\tau)=g_i(\tau+24), (i=0, 1, 2)$ we easily to find that
\begin{align*}
\stroke{G}{U_p^*}=\chi_6(p)G(p\tau)
\end{align*}
\end{proof}

From Lemma \ref{lemUpG}, we have
\begin{align}
\mylabel{eq:fuf}
\stroke{F}{U_p^*}-\chi_6(p)F(p\tau)=\stroke{H}{U_p^*}-\chi_6(p)H(p\tau).
\end{align}
This will enable us to 
find modular-properties of $f(q)$ and $\omega(q)$. 
We say a 
function $f\,:\,\mathbb{H} \longrightarrow \mathbb{C}$   
is \textit{ on $\Gamma$} if for each 
$A\in \Gamma$:
$$
\stroke{f}{A}=f;\quad\mbox{i.e.}\quad f(A\tau) = f(\tau),
$$
for all $A\in\Gamma$ and all $\tau\in\mathbb{H}$. 
We note that $H(\tau)$ is a non-holomorphic modular function.
Many properties of modular functions also hold for non-holomorphic modular 
functions. For example, Atkin and Lehner's \cite[Lemma 7]{At-Le70} holds
even if $f(\tau)$ is not holomorphic. Hence we have
\begin{lemma}
\label{lemgp}
Let $p$ be prime. 
If $f$ is on $\Gamma_0(pN)$
and $p\mid N$, then $\stroke{f}{U_p}$ is 
on $\Gamma_0(N)$.
\end{lemma}

If $e\parallel N$, we call the matrix
\begin{align*}
W_e=\left(\begin{matrix}
ae &b\\
cN &de
\end{matrix}\right), \qquad a, b, c, d\in \mathbb{Z}, \quad det(W_e)=e
\end{align*}
an Atkin-Lehner involution of $\Gamma_0(N)$.

\begin{lemma}[{\cite[Corollary 2.2]{Ch-La98}\omylabel{Ch-La98}}]
\label{lemeta}
Let $W_e$ be an Atkin-Lehner involution of $\Gamma_0(N)$. Let $t>0$ be such that $t|N$. Then
\begin{align*}
\eta(tW_e\tau)=\eta\left(t\frac{ae\tau+b}{cN\tau+de}\right)
=\nu_{\eta}(M)\left(\frac{cN\tau+de}{\delta}\right)^{1/2}\eta\left(\frac{et}{\delta^2}\tau\right),
\end{align*}
where $\delta=(e,t)$, $\nu_{\eta}$ is eta-multiplier and
\begin{align*}
M=\left(\begin{matrix}
a\delta &bt/\delta\\
cN\delta/et &de/\delta
\end{matrix}\right)\in SL_2(\mathbb{Z}).
\end{align*}
\end{lemma}

\begin{lemma}[{\cite[Lemma 6]{Ch-To10}\omylabel{Ch-To10}}]
\label{lemuw}
Let $p$ be prime, $p|N$, $e\parallel N$ and $(p,e)=1$. If $f(\tau)$ is 
on $\Gamma_0(N)$ then
\begin{align*}
\stroke{(\stroke{f}{U_p})}{W_e}=\stroke{(\stroke{f}{W_e})}{U_p}.
\end{align*}
\end{lemma}

Let
$$
\widetilde{H}=(H_0, H_1, H_2):=
\left(\frac{\eta(2\tau)^2}{\eta(\tau)^3}h_0(\tau), 
      \frac{\eta(\tau/2)^2}{2\eta(\tau)^3}h_1(\tau), 
      \frac{\eta(\tau)^3}{2\eta(\tau/2)^2\eta(\tau)^2}h_2(\tau)\right).
$$
\begin{lemma}
\label{lemh0}
$H_0(\tau)$ is on $\Gamma_0(4)$.
\end{lemma}

\begin{proof} It is well-known that
\begin{align}
\mylabel{eq:eta1}
\eta(\tau+1)=\zeta_{24}\eta(\tau),
\end{align}
and it is easy to calculate that
\begin{align}
\mylabel{eq:eta2}
\eta(\tau+1/2)=\zeta_{48}\frac{\eta(2\tau)^3}{\eta(\tau)^2\eta(4\tau)^2}.
\end{align}
From \omyeqn{htau1}, \omyeqn{eta1} and \omyeqn{eta2}, we have
\begin{align}
\mylabel{eq:htau3}
\widetilde{H}(\tau+1)=
\left(\begin{matrix}
1 &0 &0\\
0 &0 &i\\
0 &i &0
\end{matrix}\right)
\widetilde{H}(\tau).
\end{align}
It is also well-known that
\begin{align}
\mylabel{eq:eta3}
\eta(-1/\tau)=\sqrt{-i\tau}\eta(\tau).
\end{align}
From \omyeqn{htau2} and \omyeqn{eta3}, we have
\begin{align}
\mylabel{eq:htau4}
\widetilde{H}(-1/\tau)=
\left(\begin{matrix}
0 &1 &0\\
1 &0 &0\\
0 &0 &-1
\end{matrix}\right)
\widetilde{H}(\tau).
\end{align}
By \cite[Proposition 4]{Ra77},
\begin{align*}
\left(\begin{matrix}
-1 &0\\
0 &-1
\end{matrix}\right),
\left(\begin{matrix}
1 &1\\
0 &1
\end{matrix}\right),
\left(\begin{matrix}
1 &-1\\
4 &-3
\end{matrix}\right)
\text{ and }
\left(\begin{matrix}
3 &-1\\
4 &-1
\end{matrix}\right),
\end{align*}
generate $\Gamma_0(4)$. From \omyeqn{htau3} and \omyeqn{htau4} we can compute that
\begin{align*}
&H_0(\tau+1)=H_0(\tau),\\
&H_0\left(\frac{\tau-1}{4\tau-3}\right)=H_1\left(-\frac{4\tau-3}{\tau-1}\right)=H_1\left({-\frac{1}{\tau-1}}\right)=H_0(\tau-1)=H_0(\tau),\\
&H_0\left(\frac{3\tau-1}{4\tau-1}\right)=H_1\left(-\frac{4\tau-1}{3\tau-1}\right)=-iH_2\left({-\frac{\tau}{3\tau-1}}\right)=iH_2\left(-\frac{1}{\tau}+3\right)=H_0(\tau),
\end{align*}
which implies that $H_0(\tau)$ is on $\Gamma_0(4)$.
\end{proof}

From Lemma \ref{lemh0}, we can prove the following theorem.

\begin{theorem}
\omylabel{thm:f0Up}
For each prime $p\geq 5$,
\begin{align*}
\frac{\eta(2p\tau)^2}{\eta(p\tau)^3}\left(\stroke{f_0}{U_p^*}-\chi_6(p)f_0(p\tau)\right),
\end{align*}
is a weakly holomorphic modular function on $\Gamma_0(4p)$.
\end{theorem}
\begin{proof}
From Lemma \ref{lemh0} 
$$
H_0(\tau)=\frac{\eta(2\tau)^2}{\eta(\tau)^3}h_0(\tau)
$$ 
is 
on $\Gamma_0(4)$. 
Also by Theorem \mythm{etamodthm} the product
$$
\frac{\eta(2p^2\tau)^2\eta(\tau)^3}{\eta(p^2\tau)^3\eta(2\tau)^2}
$$ 
is a modular function on $\Gamma_0(2p^2)$. 
This implies 
$$
\frac{\eta(2p^2\tau)^2}{\eta(p^2\tau)^3}h_0(\tau)
$$ 
is on $\Gamma_0(4p^2)$ 
and by Lemma \ref{lemgp},
\begin{align}
\mylabel{eq:uph}
\stroke{\frac{\eta(2p^2\tau)^2}{\eta(p^2\tau)^3}h_0(\tau)}{U_p}
\end{align}
is on $\Gamma_0(4p)$. Let
\begin{align*}
A=\left(\begin{matrix}
a &b\\
c &d
\end{matrix}\right)\in \Gamma_0(4p),
\end{align*}
then
\begin{align*}
A^*=\left(\begin{matrix}
a &bp\\
c/p &d
\end{matrix}\right)\in \Gamma_0(4),
\end{align*}
and
\begin{align}
\mylabel{eq:hptau}
H_0(pA\tau)=H_0(A^*(p\tau))=H_0(p\tau).
\end{align}
By \omyeqn{fuf}, \omyeqn{uph} and \omyeqn{hptau}, 
the function  
\begin{align*}
\frac{\eta(2p\tau)^2}{\eta(p\tau)^3}\left(\stroke{f_0}{U_p^*}-\chi_6(p)f_0(p\tau)\right)=&\frac{\eta(2p\tau)^2}{\eta(p\tau)^3}\left(\stroke{h_0}{U_p^*}-\chi_6(p)h_0(p\tau)\right)\\
=&\stroke{\frac{\eta(2p^2\tau)^2}{\eta(p^2\tau)^3}h_0(\tau)}{U_p}-\chi_6(p)H_0(p\tau)
\end{align*}
is on $\Gamma_0(4p)$, and is holomorphic on $\mathbb{H}$. 
By using an argument similar to that of \cite[Section 5]{Ga19a}
we can show that the function satisfies 
condition (iii) (Section \subsect{bthy})
and is thus a weakly holomorphic function on $\Gamma_0(4p)$.
\end{proof}

For example, letting $p=5, 7$ we have
\begin{align}
\mylabel{eq:f5q}
\frac{J_{10}^2}{J_5^3}\left(\stroke{qf(q)}{U_5}+f(q^5)\right)&=\frac{J_2^4J_{10}^4}{J_1J_4^3J_5^3J_{20}}-4q\frac{J_1^2J_4^3J_{10}J_{20}}{J_2^5J_5^2},\\
\mylabel{eq:f7q}
\frac{J_{14}^2}{J_7^3}\left(\stroke{q^2f(q)}{U_7}-f(q^7)\right)&=-\frac{J_1^3J_7^3}{J_2^5J_{14}}-6q^2\frac{J_1^4J_{14}^6}{J_2^6J_7^4}.
\end{align}
Note that \omyeqn{f5q} can also be found in \cite[Eq. (3.1)]{Ch-Ch-Ga20}.

For each prime $p\geq 5$,
\begin{align*}
W(p):=\left(\begin{matrix}
p^2-1 &-1\\
4p^2 &-4
\end{matrix}\right)
\end{align*}
is an Atkin-Lehner involution on $\Gamma_0(4p^2)$ with $a=(p^2-1)/4$, $b=-1$, $c=1$, $d=-1$ and $e=4$.

\begin{theorem}
\label{thm:th2}
For each prime $p\geq 5$
\begin{align}
\mylabel{eq:fwr}
\stroke{\frac{\eta(2p\tau)^2}{\eta(p\tau)^3}\left(\stroke{f_0(\tau)}{U_p^*}-\chi_6(p)f_0(p\tau)\right)}{W(p)}
=\frac{\eta(2p\tau)^2}{2\eta(4p\tau)^3}\left(\stroke{f_1(4\tau)}{U_p^*}-\chi_6(p)f_1(4p\tau)\right).
\end{align}
\end{theorem}

\begin{proof}
We know                               
$$
M_p=\frac{\eta(2p^2\tau)^2\eta(\tau)^3}{\eta(p^2\tau)^3\eta(2\tau)^2}
$$ 
is a modular function on $\Gamma_0(2p^2)$. 
From Lemma \ref{lemh0}, $M_pH_0(\tau)$ is 
on $\Gamma_0(4p^2)$. Thus by Lemma \ref{lemuw}
\begin{align}
\mylabel{eq:we0}
\stroke{\frac{\eta(2p\tau)^2}{\eta(p\tau)^3}(\stroke{h_0(\tau)}{U_p^*})}{W(p)}=\stroke{\left(\stroke{M_pH_0(\tau)}{U_p^*}\right)}{W(p)}=\stroke{\left(\stroke{M_pH_0(\tau)}{W(p)}\right)}{U_p^*}.
\end{align}
Using \omyeqn{htau3} and \omyeqn{htau4}, and letting $a=(p^2-1)/4$ and 
$\tau_1=4\tau$ we have
\begin{align}
\mylabel{eq:we1}
\stroke{H_0(\tau)}{W(p)}=&H_0\left(\frac{a\tau_1-1}{(4a+1)\tau_1-4}\right)=H_1\left(-\frac{(4a+1)\tau_1-4}{a\tau_1-1}\right)\\
\nonumber
=&H_1\left(-\frac{\tau_1}{a\tau_1-1}\right)=H_0\left(a-\frac{1}{\tau_1}\right)=H_1(4\tau).
\end{align}
Using Lemma \ref{lemeta} we have
\begin{align}
\mylabel{eq:we2}
\stroke{M_p}{W(p)}=\frac{\eta(2p^2\tau)^2\eta(4\tau)^3}{\eta(4p^2\tau)^3\eta(2\tau)^2}
\end{align}
\omyeqn{we0}, \omyeqn{we1} and \omyeqn{we2} gives
\begin{align}
\mylabel{eq:we3}
\stroke{\frac{\eta(2p\tau)^2}{\eta(p\tau)^3}(\stroke{h_0(\tau)}{U_p^*})}{W(p)}=\frac{\eta(2p\tau)^2}{2\eta(4p\tau)^3}(\stroke{h_1(4\tau)}{U_p^*}).
\end{align}
Also, it is easy to calculate that
\begin{align}
\mylabel{eq:we4}
\stroke{\frac{\eta(2p\tau)^2}{\eta(p\tau)^3}h_0(p\tau)}{W(p)}=\stroke{H_0(p\tau)}{W(p)}=H_1(4p\tau)=\frac{\eta(2p\tau)^2}{2\eta(4p\tau)^3}h_1(4\tau).
\end{align}
\omyeqn{we3} and \omyeqn{we4} give
\begin{align}
\mylabel{eq:fuf1}
\stroke{\frac{\eta(2p\tau)^2}{\eta(p\tau)^3}\left(\stroke{h_0(\tau)}{U_p^*}-\chi_6(p)h_0(p\tau)\right)}{W(p)}=\frac{\eta(2p\tau)^2}{2\eta(4p\tau)^3}\left(\stroke{h_1(4\tau)}{U_p^*}-\chi_6(p)h_1(4p\tau)\right).
\end{align}
\omyeqn{fuf} and \omyeqn{fuf1} complete the proof.
\end{proof}

\subsection{The congruences for $\omega(q)$}

Theorem \omythm{th2} implies that if there are congruences for the 
coefficients of $f(q)$, there will be congruences for the coefficients of 
$\omega(q)$. For example, letting $p=5$ in \omyeqn{fwr}, we have
\begin{align}
\mylabel{eq:fwwe}
\stroke{\frac{J_{10}^2}{J_5^3}\left(\stroke{qf(q)}{U_5}+f(q^5)\right)}{W(5)}=\frac{J_{10}^2}{J_{20}^3}\left(\stroke{q^{-7}\omega(q^2)}{U_5}+q^5\omega(q^{10})\right).
\end{align}
Applying $W(5)$ to both sides of \omyeqn{f5q} and using \omyeqn{fwwe} and 
Lemma \ref{lemeta}, we obtain the generating function of 
$\stroke{q^{-2}\omega(q^2)}{U_5}+q^6\omega(q^{10})$ after 
multiplying both sides by $\frac{qJ_{20}^3}{J_{10}^2}$:
\begin{align}
\mylabel{eq:genw}
\sum_{n=0}^{\infty}(a_\omega(5n+1)+a_\omega((n-3)/5))q^{2n}=\frac{J_2^4J_{10}^2}{J_1^3J_4J_5}+\frac{J_1^3J_4^2J_5J_{20}}{J_2^5J_{10}}.
\end{align}

\begin{theorem}
For all $\alpha\ge3$ and all $n\ge 0$ we have
\begin{align*}
a_\omega(5^{\alpha}n + \delta_\alpha)
+ a_\omega(5^{\alpha-2}n + \delta_{\alpha-2})
\equiv 0 \pmod{5^{ \FL{\tfrac{1}{2}\alpha }}},
\end{align*}
  where $\delta_\alpha$ satisfies $0 < \delta_\alpha < 5^\alpha$ and
$3\delta_\alpha+2\equiv0\pmod{5^\alpha}$.
\end{theorem}

\begin{proof}
We define    
$$
a:=\frac{J_{50}^2J_4^4}{q^{12}J_{100}^4J_2^2}, \quad b:=\frac{q^4J_{100}}{J_4},
$$
and
\begin{align*}
P_a:=q\left(\frac{J_{20}J_{10}^2J_2^6}{J_5J_4^5J_1^3}+\frac{J_{20}^2J_5J_1^3}{J_{10}J_4^2J_2^3}\right),\\
P_b:=\frac{1}{q}\left(\frac{J_{10}^6J_4J_2^2}{J_{20}^5J_5^3J_1}-\frac{J_5^3J_4^2J_1}{J_{20}^2J_{10}^3J_2}\right).
\end{align*}
Let $K_0:=P_a$ and
$$
K_{2\alpha+1}=\stroke{aK_{2\alpha}}{U_5}, \quad K_{2\alpha+2}=\stroke{bK_{2\alpha+1}}{U_5}.
$$
From \omyeqn{genw} and a simple calculation which is similar to \omyeqn{L2a} 

and \omyeqn{L2a-1} in \cite{Ch-Ch-Ga20}, we have
\begin{align*}
K_{2\alpha}&=\frac{qJ_{20}J_2^2}{J_4^4}\sum_{n=0}^{\infty}a(5^{2\alpha}n+\gamma_{2\alpha})q^{2n},\\
K_{2\alpha+1}&=\frac{J_{10}^2J_4}{qJ_{20}^4}\sum_{n=0}^{\infty}a(5^{2\alpha+1}n+\gamma_{2\alpha+1})q^{2n},
\end{align*}
where $a(n):=a_\omega(5n+1)+a_\omega((n-3)/5)$, $\gamma_{2\alpha}=\frac{1}{3}(5^{2\alpha}-1)$ and $\gamma_{2\alpha+1}=\frac{1}{3}(2\cdot 5^{2\alpha+1}-1)$. Let
$$
t_\omega:=\frac{J_{10}^4J_4^2}{J_{20}^2J_2^4}.
$$
Using Lemma \ref{lemeta}, it is easy to see that
$$
\stroke{A}{W(5)}=a, \quad \stroke{B}{W(5)}=b,
$$
and
$$
\stroke{P_A}{W(5)}=P_a, \quad \stroke{P_B}{W(5)}=P_b, \quad \stroke{t}{W(5)}=t_\omega.
$$
We will prove that for each $\alpha\ge 0$
\begin{align}
\mylabel{eq:lkw}
K_{\alpha}=\stroke{L_{\alpha}}{W(5)}.
\end{align}
First, $K_0=P_a=\stroke{P_A}{W(5)}=\stroke{L_0}{W(5)}$, and then assume that \omyeqn{lkw} holds for $2\alpha$, by Lemma \ref{lemuw}
$$
K_{2\alpha+1}=\stroke{K_{2\alpha}}{U_5}=\stroke{\stroke{L_{\alpha}}{W(5)}}{U_5}=\stroke{\stroke{L_{\alpha}}{U_5}}{W(5)}=\stroke{L_{2\alpha+1}}{W(5)},
$$
which means \omyeqn{lkw} holds for $2\alpha+1$. Similarly, \omyeqn{lkw} holds for $2\alpha+1$ also implies that \omyeqn{lkw} holds for $2\alpha+2$. Inductively, \omyeqn{lkw} hold for each $\alpha \ge 0$. Hence
\beq
\mylabel{eq:k2a}
K_{2\alpha}=\stroke{L_{\alpha}}{W(5)}\in 5^{\alpha}X_a,
\eeq
and
\beq
\mylabel{eq:k2a+1}
K_{2\alpha+1}=\stroke{L_{\alpha+1}}{W(5)}\in 5^{\alpha+1}X_b,
\eeq
where
\begin{align*}
X_a:&=\left\{P_a\sum_{k=1}^{\infty}r(k)5^{[\frac{3k-3}{4}]}t_\omega^k,\quad \text{ r is discrete function}\right\},\\
X_b:&=\left\{P_b\sum_{k=2}^{\infty}r(k)5^{[\frac{3k-6}{4}]}t_\omega^k,\quad \text{r is discrete function}\right\}.
\end{align*}
\omyeqn{k2a} and \omyeqn{k2a+1} imply that
$$
a(5^{\alpha-1}n+\gamma_{\alpha-1})\equiv 0\pmod{5^{ \FL{\tfrac{1}{2}\alpha}}},
$$
so that  
\begin{align*}
a_\omega(5^{\alpha}n + \delta_\alpha)
+ a_\omega(5^{\alpha-2}n + \delta_{\alpha-2})
\equiv 0 \pmod{5^{ \FL{\tfrac{1}{2}\alpha }}}.
\end{align*}
This completes the proof of \omyeqn{wrmod5}. 
The proof of \omyeqn{wrmod7} is analogous.
\end{proof}

\appendix

\section{The Fundamental Relations for the rank parity function for powers of $7$}
\label{funr-7}

\begin{align*}
\text{Group \uppercase\expandafter{\romannumeral1}}&\\
&U_A(1)=8\cdot7^9t^6+176\cdot7^7t^5+16\cdot7^6t^3+1464\cdot7^5t^4+1199\cdot7^2t^2+9\cdot7^2t\\
&\qquad\quad+p_0(7^{11}t^7+23\cdot7^9t^6+206\cdot7^7t^5+125\cdot7^6t^4+242\cdot7^4t^3+23\cdot7^3t^2\\
&\qquad\quad+10t)+p_1(-7^{10}t^7-23\cdot7^8t^6-198\cdot7^6t^5-109\cdot7^5t^4-170\cdot7^3t^3\\
&\qquad\quad-9\cdot7^2t^2),\\
&U_A(t^{-1})=t,\\
&U_A(t^{-2})=-6\cdot7^2t^2-15\cdot7t+1+3\cdot7p_0t+p_1(2\cdot7^2t^2-t),\\
&U_A(t^{-3})=7^4t^3+7^4t^2+15\cdot7^2t-7-3\cdot7^2p_0t+p_1(-2\cdot7^3t^2+7t),\\
&U_A(t^{-4})=7^6t^4-44\cdot7^3t^2-660\cdot7t+44+132\cdot7p_0t+p_1(88\cdot7^2t^2-44t),\\
&U_A(t^{-5})=7^8t^5+5\cdot7^5t^2+75\cdot7^3t-5\cdot7^2-15\cdot7^3p_0t+p_1(-10\cdot7^4t^2+5\cdot7^2t),\\
&U_A(t^{-6})=7^{10}t^6+12\cdot7^9t^5+4\cdot7^9t^4+164\cdot7^6t^3+207\cdot7^4t^2-2207\cdot7^2t+23\cdot7^2\\
&\qquad\qquad+p_0(-2\cdot7^9t^5-30\cdot7^7t^4-22\cdot7^6t^3-38\cdot7^4t^2+477\cdot7^2t+2)\\
&\qquad\qquad+p_1(-2\cdot7^{10}t^6-32\cdot7^8t^5-26\cdot7^7t^4-8\cdot7^6t^3+292\cdot7^3t^2-23\cdot7^2t).\\
\text{Group \uppercase\expandafter{\romannumeral2}}&\\
&U_A(p_0t^{-1})=8\cdot7^{11}t^6+176\cdot7^9t^5+16\cdot7^8t^3+1464\cdot7^7t^4+8392\cdot7^3t^2+3072\cdot7t\\
&\qquad\qquad\quad+p_0(7^{13}t^7+23\cdot7^{11}t^6+206\cdot7^9t^5+125\cdot7^8t^4+242\cdot7^6t^3+23\cdot7^5t^2\\
&\qquad\qquad\quad+23\cdot7^5t^2+512t)+p_1(-7^{12}t^7-23\cdot7^{10}t^6-198\cdot7^8t^5\\
&\qquad\qquad\quad-109\cdot7^7t^4-170\cdot7^5t^3-3072\cdot7t^2),\\
&U_A(p_0t^{-2})=-6\cdot7^3t^2-90\cdot7t+6+p_0(7^2t^2+18\cdot7t)+p_1(7^3t^3+13\cdot7^2t^2-6t),\\
&U_A(p_0t^{-3})=-2\cdot7^3t^2-30\cdot7t+2+p_0(7^4t^3+6\cdot7t)\\
&\qquad\qquad\quad+p_1(7^5t^4+7^4t^3+4\cdot7^2t^2-2t),\\
&U_A(p_0t^{-4})=22\cdot7^3t^2+330\cdot7t-22+p_0(7^6t^4-66\cdot7t)\\
&\qquad\qquad\quad+p_1(7^7t^5+7^6t^4-44\cdot7^2t^2+22t),\\
&U_A(p_0t^{-5})=6\cdot7^9t^5+2\cdot7^9t^4+82\cdot7^6t^3+156\cdot7^4t^2-316\cdot7^2t+4\cdot7^2\\
&\qquad\qquad\quad+p_0(-6\cdot7^8t^5-15\cdot7^7t^4-11\cdot7^6t^3-19\cdot7^4t^2+81\cdot7^2t+1)\\
&\qquad\qquad\quad+p_1(-6\cdot7^9t^6-15\cdot7^8t^5-13\cdot7^7t^4-4\cdot7^6t^3+41\cdot7^3t^2-4\cdot7^2t),\\
&U_A(p_0t^{-6})=-6\cdot7^9t^5-2\cdot7^9t^4-82\cdot7^6t^3+50\cdot7^4t^2+3406\cdot7^2t-234\cdot7\\
&\qquad\qquad\quad+p_0(7^{10}t^6+7^9t^5+15\cdot7^7t^4+11\cdot7^6t^3+19\cdot7^4t^2-699\cdot7^2t-1)\\
&\qquad\qquad\quad+p_1(7^{11}t^7+2\cdot7^{10}t^6+16\cdot7^8t^5+13\cdot7^7t^4+4\cdot7^6t^3-453\cdot7^3t^2\\
&\qquad\qquad\quad+234\cdot7t),\\
&U_A(p_0t^{-7})=-510\cdot7^9t^5-170\cdot7^9t^4-6970\cdot7^6t^3-17446\cdot7^4t^2-35930\cdot7^2t\\
&\qquad\qquad\quad+258\cdot7^2+p_0(7^{12}t^7+85\cdot7^9t^5+1275\cdot7^7t^4+935\cdot7^6t^3\\
&\qquad\qquad\quad+1615\cdot7^4t^2+5673\cdot7^2t-85)+p_1(7^{13}t^8+7^12t^7+85\cdot7^{10}t^6\\
&\qquad\qquad\quad+1360\cdot7^8t^5+1105\cdot7^7t^4+340\cdot7^6t^3+4887\cdot7^3t^2-258\cdot7^2t).\\
\text{Group \uppercase\expandafter{\romannumeral3}}&\\
&U_A(p_1)=-8\cdot7^{10}t^6-176\cdot7^8t^5-16\cdot7^7t^3-1464\cdot7^6t^4-8392\cdot7^2t^2-3072t\\
&\qquad\qquad+p_0(-7^{12}t^7-23\cdot7^{10}t^6-206\cdot7^8t^5-125\cdot7^7t^4-242\cdot7^5t^3-23\cdot7^4t^2\\
&\qquad\qquad-73t)+p_1(7^{11}t^7+23\cdot7^9t^6+198\cdot7^7t^5+109\cdot7^6t^4+170\cdot7^4t^3\\
&\qquad\qquad+439\cdot7t^2),\\
&U_A(p_1t^{-1})=-7p_1t,\\
&U_A(p_1t^{-2})=6\cdot7^3t^2+90\cdot7t-6-18\cdot7p_0t+p_1(-7^3t^3-12\cdot7^2t^2+6t),\\
&U_A(p_1t^{-3})=-38\cdot7^3t^2-570\cdot7t+38+144\cdot7p_0t+p_1(-7^5t^4+76\cdot7^2t^2-38t),\\
&U_A(p_1t^{-4})=218\cdot7^3t^2+3270\cdot7t-218-654\cdot7p_0t\\
&\qquad\qquad\quad+p_1(-7^7t^5-436\cdot7^2t^2+248t),\\
&U_A(p_1t^{-5})=-6\cdot7^9t^5-2\cdot7^9t^4-82\cdot7^6t^3-340\cdot7^4t^2-2444\cdot7^2t+156\cdot7\\
&\qquad\qquad\quad+p_0(7^9t^5+15\cdot7^7t^4+11\cdot7^6t^3+19\cdot7^4t^2+471\cdot7^2t-1)\\
&\qquad\qquad\quad+p_1(6\cdot7^9t^6+16\cdot7^8t^5+13\cdot7^7t^4+4\cdot7^6t^3+327\cdot7^3t^2-156\cdot7t),\\
&U_A(p_1t^{-6})=21\cdot7^9t^5+18\cdot7^9t^4+738\cdot7^6t^3+46\cdot7^6t^2+9906\cdot7^2t-598\cdot7\\
&\qquad\qquad\quad+p_0(-9\cdot7^9t^5-135\cdot7^7t^4-99\cdot7^6t^3-171\cdot7^4t^2-1821\cdot7^2t+9)\\
&\qquad\qquad\quad+p_1(-7^{11}t^7-9\cdot7^{10}t^6-144\cdot7^8t^5-117\cdot7^7t^4-36\cdot7^6t^3\\
&\qquad\qquad\quad-1331\cdot7^3t^2+598\cdot7t).\\
\text{Group \uppercase\expandafter{\romannumeral4}}&\\
&U_B(t)=7^2t+7,\\
&U_B(1)=-7,\\
&U_B(t^{-1})=7^2+t^{-1},\\
&U_B(t^{-2})=7^5t^2+11\cdot7^3t-11\cdot7-11t^{-1},\\
&U_B(t^{-3})=7^7t^3-90\cdot7^3t-20\cdot7^2+90t^{-1},\\
&U_B(t^{-4})=-7^9t^4-38\cdot7^7t^3-38\cdot7^6t^2-19\cdot7^4t+209\cdot7^2-627t^{-1},\\
&U_B(t^{-5})=7^{11}t^5+46\cdot7^9t^4+874\cdot7^7t^3+874\cdot7^6t^2+1955\cdot7^4t-667\cdot7^2\\
&\qquad\qquad\quad+3795t^{-1}.\\
\text{Group \uppercase\expandafter{\romannumeral5}}&\\
&U_B(p_0t)=8\cdot7^{12}t^6+200\cdot7^{10}t^5+1984\cdot7^8t^4+9656\cdot7^6t^3+22896\cdot7^4t^2\\
&\qquad\qquad\quad+3144\cdot7^3t+632\cdot7+p_0(7^{14}t^7+26\cdot7^{12}t^6+274\cdot7^{10}t^5\\
&\qquad\qquad\quad+1464\cdot7^8t^4+4045\cdot7^6t^3+5172\cdot7^4t^2+2150\cdot7^2t+9\cdot7)\\
&\qquad\qquad\quad+p_1(-7^{13}t^7-26\cdot7^{11}t^6-38\cdot7^{10}t^5-1328\cdot7^7t^4\\
&\qquad\qquad\quad-459\cdot7^6t^3-3132\cdot7^3t^2-30\cdot7^2t),\\
&U_B(p_0)=6\cdot7p_0-7^2p_1t,\\
&U_B(p_0t^{-1})=-4\cdot7^3t+4\cdot7+4t^{-1}+p_0(7^3t-2\cdot7)+p_1(7^4t^2+7^2t-4),\\
&U_B(p_0t^{-2})=8\cdot7^5t^2+76\cdot7^3t+12\cdot7^2-20t^{-1}+p_0(-7^5t^2-10\cdot7^3t+3\cdot7)\\
&\qquad\qquad\quad+p_1(-7^6t^3-9\cdot7^4t^2+3\cdot7),\\
&U_B(p_0t^{-3})=8\cdot7^7t^3+48\cdot7^5t^2-12\cdot7^4t-44\cdot7^2+100t^{-1}+p_0(-7^7t^3-8\cdot7^5t^2\\
&\qquad\qquad\quad+8\cdot7^3t-3\cdot7^2)+p_1(-7^8t^4-7^7t^3-10\cdot7^4t^2-2\cdot7^4t-17\cdot7),\\
&U_B(p_0t^{-4})=-4\cdot7^9t^4-164\cdot7^7t^3-4\cdot7^8t^2-480\cdot7^4t-136\cdot7^2-424t^{-1}\\
&\qquad\qquad\quad+p_0(13\cdot7^7t^3+23\cdot7^6t^2+55\cdot7^4t+31\cdot7^2+t^{-1})\\
&\qquad\qquad\quad+p_1(9\cdot7^8t^4+13\cdot7^7t^3+46\cdot7^5t^2+135\cdot7^3t+94\cdot7),\\
&U_B(p_0t^{-5})=12\cdot7^{10}t^4+316\cdot7^8t^3+2540\cdot7^6t^2+7244\cdot7^4t+4092\cdot7^2+148\cdot7t^{-1}\\
&\qquad\qquad\quad+p_0(7^{11}t^5+13\cdot7^9t^4-71\cdot7^7t^3-241\cdot7^6t^2-761\cdot7^4t-286\cdot7^2\\
&\qquad\qquad\quad-13t^{-1})+p_1(7^{12}t^6+20\cdot7^{10}t^5+106\cdot7^8t^4-7^7t^3-206\cdot7^5t^2\\
&\qquad\qquad\quad-793\cdot7^3t-481\cdot7).\\
\text{Group \uppercase\expandafter{\romannumeral6}}&\\
&U_B(p_1t)=-7p_0,\\
&U_B(p_1)=7p_0+p_1(7^2t+1),\\
&U_B(p_1t^{-1})=4\cdot7^3t+12\cdot7-4t^{-1}-4\cdot7p_0+p_1(-7^4t^2-10\cdot7^2t-6),\\
&U_B(p_1t^{-2})=-8\cdot7^5t^2-100\cdot7^3t-36\cdot7^2+44t^{-1}+p_0(2\cdot7^5t^2+10\cdot7^3t+3\cdot7^2)\\
&\qquad\qquad\quad+p_1(7^6t^3+16\cdot7^4t^2+80\cdot7^2t+5\cdot7),\\
&U_B(p_1t^{-3})=-8\cdot7^7t^3+92\cdot7^4t+316\cdot7^2-356t^{-1}+p_0(-4\cdot7^6t^2-20\cdot7^4t\\
&\qquad\qquad\quad-17\cdot7^2)+p_1(7^8t^4-2\cdot7^6t^2-10\cdot7^4t-27\cdot7),\\
&U_B(p_1t^{-4})=4\cdot7^9t^4+228\cdot7^7t^3+228\cdot7^6t^2+152\cdot7^4t-240\cdot7^3+2432t^{-1}\\
&\qquad\qquad\quad+p_0(7^9t^4+15\cdot7^7t^3+7^8t^2+209\cdot7^4t+99\cdot7^2-t^{-1})\\
&\qquad\qquad\quad+p_1(-18\cdot7^8t^4-17\cdot7^7t^3+26\cdot7^5t^2+321\cdot7^3t+128\cdot7),\\
&U_B(p_1t^{-5})=-116\cdot7^9t^4-572\cdot7^8t^3-4604\cdot7^6t^2-11804\cdot7^4t+1444\cdot7^2\\
&\qquad\qquad\quad-2036\cdot7t^{-1}+p_0(-2\cdot7^{11}t^5-57\cdot7^9t^4-79\cdot7^8t^3-89\cdot7^7t^2\\
&\qquad\qquad\quad-1977\cdot7^4t-584\cdot7^2+3\cdot7t^{-1})+p_1(-7^{12}t^6-19\cdot7^{10}t^5+74\cdot7^8t^4\\
&\qquad\qquad\quad+209\cdot7^7t^3+470\cdot7^5t^2-709\cdot7^3t-465\cdot7).
\end{align*}

\section{The Fundamental Relations for the  crank parity function for powers of $7$}\label{funcr-7}

\begin{align*}
\text{Group \uppercase\expandafter{\romannumeral1}}&\\
&U^{(1)}(1)=2p_0+7p_1,\\
&U^{(1)}(t^{-1})=7^2t+p_0(-4\cdot 7^2t-4)+4\cdot 7p_1,\\
&U^{(1)}(t^{-2})=7^4t^2-4\cdot 7^2t+1+p_0(-4\cdot 7^4t^2+8\cdot 7^2t+2\cdot 7)+p_1(4\cdot 7^3t^2-13\cdot 7),\\
&U^{(1)}(t^{-3})=-8\cdot 7^4t^2+24\cdot 7^2t-2\cdot 7+p_0(2\cdot 7^6t^6+8\cdot 7^5t^2+8\cdot 7^3t-4\cdot 7)\\
&\qquad \qquad \quad+p_1(-2\cdot 7^5t^2-55\cdot 7^3 t+t^{-1}),\\
&U^{(1)}(t^{-4})=-7^8t^4-2\cdot 7^7t^3-2\cdot 7^5t^2-32\cdot 7^3t+20\cdot 7\\
&\qquad \qquad \quad+p_0(-4\cdot 7^7t^3-76\cdot 7^5t^2-144\cdot 7^3t-4\cdot 7^2)\\
&\qquad \qquad \quad+p_1(4\cdot 7^6t^2+74\cdot 7^4t+78\cdot 7^2-2\cdot 7t^{-1}),\\
&U^{(1)}(t^{-5})=-7^{10}t^5+4\cdot 7^8t^4+27\cdot 7^7t^3+116\cdot 7^5t^2+314\cdot 7^3t-27\cdot 7^2\\
&\qquad \qquad \quad+p_0(4\cdot 7^{10}t^5-64\cdot 7^8t^4+90\cdot 7^7t^3+712\cdot 7^5t^2+1504\cdot 7^3t+492\cdot 7)\\
&\qquad \qquad \quad+p_1(-4\cdot 7^9t^4-58\cdot 7^7t^3-78\cdot 7^6t^2-639\cdot 7^4t-974\cdot 7^2+135t^{-1}),\\
&U^{(1)}(t^{-6})=7^{12}t^6+44\cdot 7^{10}t^5+284\cdot 7^8t^4-118\cdot 7^7t^3-1348\cdot 7^5t^2-2740\cdot 7^3t\\
&\qquad \qquad \quad+1243\cdot 7+p_0(-88\cdot 7^{10}t^5-200\cdot 7^9t^4-204\cdot 7^8t^3-6568\cdot 7^5t^2\\
&\qquad \qquad \quad-13064\cdot 7^3t-682\cdot 7^2)+p_1(86\cdot 7^9t^4+176\cdot 7^8t^3+1130\cdot 7^6t^2\\
&\qquad \qquad \quad+734\cdot 7^5t+8679\cdot 7^2-22\cdot 7^2t^{-1}).\\
\text{Group \uppercase\expandafter{\romannumeral2}}&\\
&U^{(1)}(p_0)=7^{14}t^7+22\cdot 7^{12}t^6+190\cdot 7^{10}t^5+16\cdot 7^{10}t^4+1497\cdot 7^6t^3+1028\cdot 7^4t^2\\
&\qquad \qquad \quad+2\cdot 7^4t+p_0(8\cdot 7^{12}t^6+24\cdot 7^{11}t^5+192\cdot 7^9t^4+4888\cdot 7^6t^3\\
&\qquad \qquad \quad+7408\cdot 7^4t^2+2967\cdot 7^2t+20)+p_1(7^{15}t^7+30\cdot 7^{13}t^6+366\cdot 7^{11}t^5\\
&\qquad \qquad \quad+328\cdot 7^{10}t^4+1095\cdot 7^8t^3+12556\cdot 7^5t^2+7722\cdot 7^3t+680\cdot 7),\\
&U^{(1)}(p_0t^{-1})=5\cdot 72t+p_0(-7^4t^2-32\cdot 7^2t+62)+p_1(7^3t+80\cdot 7),\\
&U^{(1)}(p_0t^{-2})=7^5t^2+25\cdot 7^2t+p_0(7^6t^3-16\cdot 7^4t^2-96\cdot 7^2t-48)\\
&\qquad \qquad \qquad+p_1(-7^5t^2+17\cdot 7^3t+72\cdot 7),\\
&U^{(1)}(p_0t^{-3})=-7^6t^3-14\cdot 7^4t^2-4\cdot 7^3t+7+p_0(-7^8t^4-4\cdot 7^6t^3+96\cdot 7^4t^2\\
&\qquad \qquad \qquad+472\cdot 7^2t+4\cdot 7^2)+p_1(7^7t^3+3\cdot 7^5t^2-2\cdot 7^5t-51\cdot 7^2+6t^{-1}),\\
&U^{(1)}(p_0t^{-4})=-5\cdot 7^8t^4-9\cdot 7^7t^3-40\cdot 7^5t^2-66\cdot 7^3t-9\cdot 7\\
&\qquad \qquad \qquad+p_0(_7^{10}t^5+24\cdot 7^8t^4+16\cdot 7^7t^3-48\cdot 7^5t^2-362\cdot 7^3t-164\cdot 7)\\
&\qquad \qquad \qquad+p_1(-7^9t^4-23\cdot 7^7t^3-12\cdot 7^6t^2+62\cdot 7^4t+291\cdot 7^2-40t^{-1}),\\
&U^{(1)}(p_0t^{-5})=-5\cdot 7^{10}t^5-15\cdot 7^8t^4+60\cdot 7^7t^3+458\cdot 7^5t^2+786\cdot 7^3t+78\cdot 7\\
&\qquad \qquad \qquad+p_0(_7^{12}t^6+32\cdot 7^{10}t^5+44\cdot 7^9t^4+240\cdot 7^7t^3+916\cdot 7^5t^2\\
&\qquad \qquad \qquad+2416\cdot 7^3t+1004\cdot 7)+p_1(-7^{11}t^5-31\cdot 7^9t^4-272\cdot 7^7t^3\\
&\qquad \qquad \qquad-206\cdot 7^6t^2-110\cdot 7^5t-1810\cdot 7^2+216t^{-1}),\\
&U^{(1)}(p_0t^{-6})=6\cdot 7^{12}t^6+204\cdot 7^{10}t^5+240\cdot 7^9t^4+401\cdot 7^7t^3-2528\cdot 7^5t^2\\
&\qquad \qquad \qquad-5821\cdot 7^3t-96\cdot 7^2+p_0(-16\cdot 7^{12}t^6-584\cdot 7^{10}t^5-6576\cdot 7^8t^4\\
&\qquad \qquad \qquad-4885\cdot 7^7t^3-11832\cdot 7^5t^2-16512\cdot 7^3t-5736\cdot 7+t^{-1})\\
&\qquad \qquad \qquad+p_1(16\cdot 7^{11}t^5+562\cdot 7^9t^4+837\cdot 7^8t^3+82\cdot 7^8t^2+1213\cdot 7^5t\\
&\qquad \qquad \qquad+1556\cdot 7^3-136\cdot 7t^{-1}),\\
\text{Group \uppercase\expandafter{\romannumeral3}}&\\
&U^{(1)}(p_1t^{-1})=7^{15}t^7+22\cdot 7^{13}t^6+190\cdot 7^{11} t^5+16\cdot 7^{11} t^4+1497\cdot 7^7 t^3+1028\cdot 7^5 t^2\\
&\qquad \qquad \qquad+687\cdot 7^2 t+p_0(8\cdot 7^{13} t^6+24\cdot 7^{12} t^5+192\cdot 7^{10} t^4+4888\cdot 7^7 t^3\\
&\qquad \qquad \qquad+7408\cdot 7^5 t^2+20764\cdot 7^2 t+148)+p_1(30\cdot 7^{14} t^6+366\cdot 7^{12} t^5\\
&\qquad \qquad \qquad+328\cdot 7^{11} t^4+1095\cdot 7^9 t^3+12556\cdot 7^6 t^2+7722\cdot 7^4 t+4772\cdot 7),\\
&U^{(1)}(p_1t^{-2})=7^4t^2+36\cdot 7^2 t+p_0(-12\cdot 7^4 t^2-240\cdot 7^2 t+430)\\
&\qquad \qquad \qquad+p_1 (12\cdot 7^3 t+570\cdot 7),\\
&U^{(1)}(p_1t^{-3})=48\cdot 7^4 t^2+193\cdot 7^2 t+p_0 (10\cdot 7^6 t^3-64\cdot 7^4 t^2-552\cdot 7^2 t-44\cdot 7)\\
&\qquad \qquad \qquad+p_1 (-10\cdot 7^5 t^2+74\cdot 7^3 t+430\cdot 7+t^{-1}),\\
&U^{(1)}(p_1t^{-4})=-7^8 t^4-22\cdot 7^6 t^3-197\cdot 7^4 t^2-64\cdot 7^3 t+8\cdot 7\\
&\qquad \qquad \qquad+p_0 (-8\cdot 7^8 t^4-68\cdot 7^6 t^3+268\cdot 7^4 t^2+48\cdot 7^4 t+156\cdot 7)\\
&\qquad \qquad \qquad+p_1(8\cdot 7^7 t^3+61\cdot 7^5 t^2-314\cdot 7^3 t-270\cdot 7^2+34t^{-1}),\\
&U^{(1)}(p_1t^{-5})=-7^{10} t^5-36\cdot 7^8 t^4-44\cdot 7^7 t^3-144\cdot 7^5 t^2-138\cdot 7^3 t-76\cdot 7\\
&\qquad \qquad \qquad+p_0 (12\cdot 7^{10} t^5+264\cdot 7^8 t^4+234\cdot 7^7 t^3+248\cdot 7^5 t^2-1440\cdot 7^3 t\\
&\qquad \qquad \qquad-804\cdot 7)+p_1 (-12\cdot 7^9t^4-251\cdot 7^7 t^3-194\cdot 7^6 t^2-59\cdot 7^4 t\\
&\qquad \qquad \qquad+1342\cdot 7^2-225t^{-1}),\\
&U^{(1)}(p_1t^{-6})=7^{12}t^6+4\cdot 7^{10} t^5+197\cdot 7^8 t^4+414\cdot 7^7 t^3+2196\cdot 7^5 t^2+3124\cdot 7^3 t\\
&\qquad \qquad \qquad+657\cdot 7+p_0 (8\cdot 7^{12} t^6+160\cdot 7^{10} t^5+808\cdot 7^8 t^4+204\cdot 7^7 t^3\\
&\qquad \qquad \qquad+1144\cdot 7^5 t^2+8776\cdot 7^3t+638\cdot 7^2)+p_1 (-8\cdot 7^11 t^5-153\cdot 7^9 t^4\\
&\qquad \qquad \qquad-652\cdot 7^7 t^3-158\cdot 7^6 t^2-178\cdot 7^5 t-7619\cdot 7^2+1170t^{-1}),\\
&U^{(1)}(p_1t^{-7})=-7^{14} t^7-2\cdot 7^{12} t^6+563\cdot 7^{10} t^5+740\cdot 7^9 t^4+676\cdot 7^7 t^3\\
&\qquad \qquad \qquad-12972\cdot 7^5 t^2-26004\cdot 7^3 t-776\cdot 7^2+p_0(2\cdot 7^{14} t^7-88\cdot 7^{12} t^6\\
&\qquad \qquad \qquad-3180\cdot 7^{10} t^5-4496\cdot 7^9t^4-20086\cdot7^7 t^3-40840\cdot 7^5 t^2\\
&\qquad \qquad \qquad-60792\cdot 7^3 t-3340\cdot 7^2+10t^{-1})+p_1(-2\cdot 7^{13} t^6+13\cdot 7^{12} t^5\\
&\qquad \qquad \qquad+3084\cdot 7^9 t^4+3986\cdot 7^8 t^3+16260\cdot 7^6t^2+28524\cdot 7^4 t+43476\cdot 7^2\\
&\qquad \qquad \qquad-2\cdot 7^4t^{-1}).\\
\text{Group \uppercase\expandafter{\romannumeral4}}&\\
&U^{(0)}(1)=1,\\
&U^{(0)}(t^{-1})=-7t-4,\\
&U^{(0)}(t^{-2})=-7^3t^2+20,\\
&U^{(0)}(t^{-3})=-7^5t^3-88,\\
&U^{(0)}(t^{-4})=-7^7t^4+260,\\
&U^{(0)}(t^{-5})=-7^9t^5+68\cdot 7,\\
&U^{(0)}(t^{-6})=-7^11t^6-2392\cdot 7.\\
\text{Group \uppercase\expandafter{\romannumeral5}}&\\
&U^{(0)}(p_0)=7^9 t^5+11\cdot 7^7 t^4+38\cdot 7^5 t^3+31\cdot 7^3 t^2+6\cdot 7 t\\
&\qquad \qquad \qquad+p_0 (8\cdot 7^7 t^4+80\cdot 7^5 t^3+216\cdot 7^3 t^2+79\cdot 7t)\\
&\qquad \qquad \qquad+p_1 (7^{10} t^5+19\cdot 7^8 t^4+18\cdot 7^7 t^3+327\cdot 7^4 t^2+34\cdot 7^3 t+4),\\
&U^{(0)}(p_0t^{-1})=-7^2 t+1+p_0(-7^3 t^2-8\cdot 7t)+p_1(7^2t+7),\\
&U^{(0)}(p_0t^{-2})=-7^4t^2-7^3 t-12+p_0(-7^5 t^3-8\cdot 7^3 t^2)+p_1(7^4t^2+7^3t),\\
&U^{(0)}(p_0t^{-3})=2\cdot 7^6 t^3+24\cdot 7^4 t^2+83\cdot 7^2 t+108+p_0(6\cdot 7^7t^4+62\cdot 7^5t^3\\
&\qquad \qquad \qquad+27\cdot 7^4 t^2+10\cdot 7^2 t+1)+p_1(-6\cdot 7^6 t^3-8\cdot 7^5 t^2-3\cdot 7^4 t-4\cdot 7),\\
&U^{(0)}(p_0t^{-4})=-7^8 t^4-25\cdot 7^6 t^3-186\cdot 7^4t^2-498\cdot 7^2 t-120\cdot 7\\
&\qquad \qquad \qquad+p_0(-7^9 t^5-50\cdot 7^7 t^4-60\cdot 7^6 t^3-162\cdot 7^4 t^2-60\cdot 7^2 t-6)\\
&\qquad \qquad \qquad+p_1(7^8\cdot t^4+7^8 t^3+54\cdot 7^5 t^2+18\cdot 7^4 t+24\cdot 7),\\
&U^{(0)}(p_0t^{-5})=-7^{10}t^5-7^9 t^4+9\cdot 7^7 t^3+93\cdot 7^5 t^2+249\cdot 7^3 t+836\cdot 7\\
&\qquad \qquad \qquad+p_0 (-7^{11} t^6-8\cdot 7^9 t^5+3\cdot 7^9 t^4+30\cdot 7^7 t^3+81\cdot 7^5 t^2+30\cdot 7^3 t\\
&\qquad \qquad \qquad+3\cdot 7)+p_1(7^{10}t^5+7^9t^4-3\cdot 7^8 t^3-27\cdot 7^6t^2-9\cdot 7^5t-12\cdot 7^2),\\
&U^{(0)}(p_0t^{-6})=-7^{12}t^6-7^{11}t^5+6\cdot 7^7t^3+62\cdot 7^5t^2+166\cdot 7^3t-748\cdot 7^2\\
&\qquad \qquad \qquad+p_0(-7^{13}t^7-8\cdot 7^{11}t^6+2\cdot 7^9t^4+20\cdot 7^7t^3+54\cdot 7^5t^2+20\cdot 7^3t\\
&\qquad \qquad \qquad+2\cdot 7)+p_1(7^{12}t^6+7^{11}t^5-2\cdot 7^8t^3-18\cdot 7^6t^2-6\cdot 7^5t-8\cdot 7^2),\\
\text{Group \uppercase\expandafter{\romannumeral6}}&\\
&U^{(0)}(p_1t^{-1})=7^{10}t^5+11\cdot 7^8t^4+38\cdot 7^6t^3+31\cdot 7^4t^2+41\cdot 7t-1,\\
&\qquad \qquad \qquad+p_0(8\cdot 7^8t^4+80\cdot 7^6t^3+216\cdot 7^4t^2+552\cdot 7t)\\
&\qquad \qquad \qquad+p_1(7^{11}t^5+19\cdot 7^9t^4+\cdot 7^8t^3+\cdot 7^5t^2+\cdot 7^4t+29),\\
&U^{(0)}(p_1t^{-2})=-7^3 t^2-8\cdot 7^2 t+11+p_0 (-8\cdot 7^3 t^2-8\cdot 7^2 t)+p_1 (8\cdot 7^2 t+ 7^2),\\
&U^{(0)}(p_1t^{-3})=-7^5 t^3-8\cdot 7^4 t^2-7^4 t-96+p_0 (-8\cdot 7^5 t^3-8\cdot 7^4 t^2)\\
&\qquad \qquad \qquad+p_1 (8\cdot 7^4 t^2+ 7^4 t),\\
&U^{(0)}(p_1t^{-4})=-7^7 t^4+16\cdot 7^6 t^3+199\cdot 7^4 t^2+664\cdot 7^2 t+740\\
&\qquad \qquad \qquad+p_0 (48\cdot 7^7 t^4+72\cdot 7^6 t^3+216\cdot 7^4 t^2+80\cdot 7^2 t+8)\\
&\qquad \qquad \qquad+p_1 (-48\cdot 7^6 t^3-65\cdot 7^5 t^2-24\cdot 7^4 t-32\cdot 7),\\
&U^{(0)}(p_1t^{-5})=-7^9 t^5-8\cdot 7^8 t^4-31\cdot 7^7 t^3-248\cdot 7^5 t^2-664\cdot 7^3 t-740\cdot 7\\
&\qquad \qquad \qquad+p_0 (-8\cdot 7^9 t^5-64\cdot 7^8 t^4-80\cdot 7^7 t^3-216\cdot 7^5 t^2-80\cdot 7^3 t-8\cdot 7)\\
&\qquad \qquad \qquad+p_1 (8\cdot 7^8 t^4+9\cdot 7^8 t^3+72\cdot 7^6 t^2+24\cdot 7^5 t+32\cdot 7^2),\\
&U^{(0)}(p_1t^{-6})=-7^{11} t^6-8\cdot 7^{10} t^5-7^{10} t^4+120\cdot 7^7 t^3+1240\cdot 7^5 t^2+3320\cdot 7^3 t\\
&\qquad \qquad \qquad+4720\cdot 7+p_0 (-8\cdot 7^{11} t^6-8\cdot 7^{10} t^5+40\cdot 7^9 t^4+400\cdot 7^7 t^3\\
&\qquad \qquad \qquad+1080\cdot 7^5 t^2+400\cdot 7^3 t+40\cdot 7)+p_1 (8\cdot 7^{10} t^5+7^{10} t^4-40\cdot 7^8 t^3\\
&\qquad \qquad \qquad-360\cdot 7^6 t^2-120\cdot 7^5 t-160\cdot 7^2),\\
&U^{(0)}(p_1t^{-7})=6\cdot 7^{13} t^7+20\cdot 7^{12} t^6+272\cdot 7^{10} t^5+1893\cdot 7^8 t^4+435\cdot 7^7 t^3\\
&\qquad \qquad \qquad-2953\cdot 7^5 t^2-10620\cdot 7^3 t-27029\cdot 7+p_0 (-8\cdot 7^{13} t^7-8\cdot 7^{12} t^6\\
&\qquad \qquad \qquad-136\cdot 7^9 t^4-1360\cdot 7^7 t^3-3672\cdot 7^5 t^2-1360\cdot 7^3 t-136\cdot 7)\\
&\qquad \qquad \qquad+p_1 (8\cdot 7^{12} t^6+6\cdot 7^{11} t^5-19\cdot 7^9 t^4+818\cdot 7^7 t^3+1167\cdot 7^6 t^2\\
&\qquad \qquad \qquad+2798\cdot 7^4 t+11\cdot 7^4+t^{-1}).
\end{align*}






\end{document}